\documentclass{amsart}
\usepackage[english]{babel}
\usepackage{amssymb,enumerate}
\usepackage{fancyhdr}
\usepackage{geometry}
\newcommand{\assign}{:=}
\newcommand{\backassign}{=:}
\newcommand{\infixand}{\text{ and }}
\newcommand{\nequiv}{\mathrel{\not\equiv}}
\newcommand{\nobracket}{}

\newcommand{\tmem}[1]{{\em #1\/}}
\newcommand{\tmop}[1]{\ensuremath{\operatorname{#1}}}
\newcommand{\tmstrong}[1]{\textbf{#1}}
\newcommand{\tmtextbf}[1]{\text{{\bfseries{#1}}}}
\newcommand{\tmtextit}[1]{\text{{\itshape{#1}}}}
\newenvironment{enumeratealpha}{\begin{enumerate}[a{\textup{)}}] }{\end{enumerate}}
\newenvironment{enumeratenumeric}{\begin{enumerate}[1.] }{\end{enumerate}}

\newtheorem{theorem}{Theorem}[section]
\newtheorem{corollary}[theorem]{Corollary}
\newtheorem{definition}[theorem]{Definition}
\newtheorem{lemma}[theorem]{Lemma}
\newtheorem{proposition}[theorem]{Proposition}
\theoremstyle{remark}\newtheorem{remark}{Remark}[section]
\theoremstyle{remark}\newtheorem{conjecture}{Conjecture}[section]

\makeatletter
\@addtoreset{equation}{section}
\makeatother

\begin{document}

\title{On the subadditivity of generalized Kodaira dimensions}

\author{Bojie he and Xiangyu Zhou}

\address{Bojie He: School of Mathematical Sciences, Peking University, Beijing, 100871, P. R. China}

\email{hbj@amss.ac.cn}

\address{Xiangyu Zhou: Institute of Mathematics, Academy of Mathematics and Systems Science, Chinese Academy of Sciences, Beijing, 100190, P. R. China}

\email{xyzhou@math.ac.cn}

\thanks{The second author is partially supported by the National Natural Science Foundation of China (grant no. 11688101 and no. 12288201)}



\begin{abstract}
     The goals of this paper are of two aspects. Firstly, we introduce the notion of generalized numerical Kodaira dimension with multiplier ideal sheaf and establish the subadditivity inequalities in terms of this notion, which can be used to give an analytic proof of O. Fujino's result on the subadditivity of the log Kodaira dimensions. Secondly, motivated by Zhou-Zhu's subadditivity of generalized Kodaira dimensions, we adopt another definition of generalized Kodaira dimension with multiplier ideal sheaf and show they are equal by using Okounkov bodies. As one application, we show that the superadditivity part in Zhou-Zhu's setting also holds true. As another application, we give an alternative proof of Zhou-Zhu's subadditivity formula, in the case when the singular metric $h_L$ has analytic singularities, by using generalized Iitaka fibrations.
\end{abstract}

{\maketitle}

\section{Introduction and main results}

\subsection{Introduction and background}

Let $X$ and $Y$ be two compact connected complex manifolds of dimension $n$ and $m$
($n \geq m$). In the analytic setting, we call $f : X  \rightarrow Y $ an
{\tmstrong{(analytic) fiber space}}, if $f$ is a surjective holomorphic
mapping with connected fibers. Moreover, when $X$ and $Y$ are supposed to be
projective manifolds (resp. K\"{a}hler manifolds), the fiber space $f : X \rightarrow Y$ is usually called
an {\tmstrong{algebraic fiber space}} (resp. {\tmstrong{K\"{a}hler fiber space}}).

\ Let
\[ \kappa (X) \assign \max \{\nu \in\mathbb{Z}\bigcup\{-\infty\}; \limsup_{k \rightarrow \infty} \frac{ h^0 (X,
   \mathcal{O}_X (k K_X))}{k^{\nu}}>0\} \in \{ - \infty, 0, 1, \ldots, n \} \]
be the Kodaira dimension of $X$, where $K_X$ is the canonical divisor. When
$\kappa (X) = \dim X = n$, $X$ is said to be of general type.

\begin{conjecture}[Conjecture $C_{n, m}$ or Iitaka conjecture]
Let $f : X \rightarrow Y$ be an algebraic fiber space
between two projective manifolds $X$ and $Y$. Let $F$ denote the general fiber
of $f$, then
\begin{equation}
 \kappa (X) \geq \kappa (F) + \kappa (Y) \label{c,n,m} .
\end{equation}
\end{conjecture}

Note that when $X$ is not assumed to be K{\"a}hler, inequality (\ref{c,n,m}) does
not always hold in general (cf. Remark 15.3 of {\cite{ueno00}}).

In {\cite{viehweg1}}, E. Viehweg proposed a stronger version of Iitaka
conjecture:

\begin{conjecture}[Conjecture $C_{n, m}^+$ or generalized Iitaka conjecture]
Let $f :
X \rightarrow Y$ be an algebraic fiber space between two projective manifolds
$X$ and $Y$. Let $F$ denote the general fiber of $f$, then{\medskip}
\begin{equation}
 \kappa (X) \geq \kappa (F) + \max \{ \tmop{Var} (f), \kappa (Y) \} .
 \label{c,n,m+}
\end{equation}
\end{conjecture}

Here $\tmop{Var} (f)$, roughly speaking, stands for the dimension of the
birational equivalence class of all the general fibers $F$.

\

Let $L$ (resp. $D$) be a $\mathbb{Q}$-line bundle (resp.  $\mathbb{Q}$-Cartier divisor)
over a compact (connected) complex manifold $X$, its \textbf{Iitaka $L$(resp. $D$)-dimension} $\kappa
(X, L)$ (resp. $\kappa (X, D)$) is usually defined as
\[ \kappa (X, L) \assign \max\{ \nu\in\mathbb{Z}\bigcup \{-\infty\}; \limsup_{k \rightarrow \infty} \frac{ h^0 (X, k
   k_0 L)}{ k^{\nu}}>0 \} \in \{- \infty, 0, \ldots, \dim X\} \]
\[ (\tmop{resp}.\quad \kappa (X, D) \assign \max\{ \nu\in\mathbb{Z}\bigcup \{-\infty\}; \limsup_{k \rightarrow \infty} \frac{ h^0 (X,  \mathcal{O}_{X}(k k_0 D))}{ k^{\nu}}>0 \} \]

   \[ \in \{- \infty, 0, \ldots, \dim X\} ),\]
where $k_0$ is the smallest integer such that $k_0 L$ (resp. $\mathcal{O}_X
(k_0 D)$) is a line bundle. One may see in section \ref{sect5} for some basic
and further explanations of this notion.

\fancyhead[LE]{\scriptsize \thepage} 
\fancyhead[CE]{\footnotesize BOJIE HE AND XIANGYU ZHOU}
\fancyhead[CO]{\footnotesize SUBADDITIVITY OF GENERALIZED KODAIRA DIMENSIONS}
\fancyhead[RO]{\scriptsize \thepage}

Assume that $X$ is projective and let $A$ be an ample line bundle on
$X$. By perturbing all sections of $k k_0 L$ with $A$,
Nakayama's \textbf{numerical dimension} of $D$ or $L = \mathcal{O} (D)$ (cf. Definition
2.7 in {\cite{fujino1}}, Chapter V in {\cite{nakayama}} or Definition 2.4.8 in
{\cite{fujino-book}}) is defined as
\[ \kappa_{\sigma} (X, L) \assign \max_{E : \text{a divisor of $X$}}
   \{\kappa_{\sigma} (X, L ; E)\} = \max_{m \in \mathbb{N}} \{\kappa_{\sigma}
   (X, L ; m A)\}, \]
where $\kappa_{\sigma} (X, L ; E)$ is defined to be
\[ \max \{\nu \in \mathbb{Z} \cup \{- \infty\}; \limsup_{k \rightarrow \infty}
   \frac{h^0 (X, k k_0 L + E)}{k^{\nu}} > 0\} . \]
\quad There are some basic properties about $\kappa$ and
$\kappa_{\sigma}$. For example, $\kappa (X, L) \geq 0$ if and only if $m_0 L$ is
effective for some $m_0 \in\mathbb{N}$, $\kappa_{\sigma} (X, L) \geq 0$ if and only if $L$ is
pseudoeffective. Some other properties are e.g. $\kappa (X, L) \leq \kappa_{\sigma} (X, L)$, $\kappa_{\sigma}
(X, L)$ equals the numerical dimension $\nu (X, L)$ when $L$ is nef, etc. To
see other various useful characterizations of $\kappa$ and $\kappa_{\sigma}$,
one may refer to {\cite{nakayama}} or {\cite{fujino-book}}.

In case when $L = K_X$ is the canonical bundle, $\kappa(X,K_X)$ (resp. $\kappa_{\sigma}(X,K_X)$)
is called the Kodaira dimension (resp. numerical Kodaira dimension) of $X$.

When $L$ is moreover assumed to be pseudoeffective, i.e. which can be equipped
with a singular metric $h_L$ whose curvature current is semi-positive, one may
similarly define the{\tmstrong{ {\tmstrong{{\tmstrong{generalized}} Kodaira
dimension}}}} (cf. Definition 1.1 in {\cite{zhou-zhu1}}) $\kappa (X, K_X + L,
h_L)$ as
\begin{equation}
  \max \{\nu \in \mathbb{Z}\bigcup \{-\infty\} ; \limsup_{k \rightarrow \infty} \frac{h^0 (X, k
  k_0 (K_X + L) \otimes \mathcal{I}_{k k_0} (h_L))}{k^{\nu}} > 0\} \label{13}
\end{equation}
and the {\tmstrong{generalized numerical Kodaira dimension}} $\kappa_{\sigma}
(X,  K_X + L, h_L)$ as
\[ \max_{E : \text{a divisor of $X$}} \{\kappa_{\sigma} (X, K_X + L, h_L ;
   E)\} = \max_{k \in \mathbb{N}} \{\kappa_{\sigma} (X, K_X + L,h_L ; k A)\}, \]
where $\kappa_{\sigma} (X, K_X + L, h_L ; E)$ is defined to be
\begin{equation}
  \max \{\nu \in \mathbb{Z} \bigcup\{-\infty\} ; \limsup_{k \rightarrow \infty} \frac{h^0 (X, (k
  k_0 (K_X + L) + E) \otimes \mathcal{I}_{k k_0} (h_L))}{k^{\nu}} > 0\} .
\end{equation}
Here the $k k_0$-multiplier ideal sheaves $\mathcal{I }_{k k_0} (h_L)$ are
precisely defined as in (\ref{ddjo}).

The definition of generalized Kodaira dimensions with multiplier ideal
sheaves is motivated by the study of the graded subalgebra $S=\bigoplus_{k\in\mathbb{N}}S_k \subset R (X,
K_X + L)$ with $S_k=H^0(X,k k_0(K_X +L)\otimes\mathcal{I}_{k k_0}(h_L))$ (cf. (\ref{notj})), according to the very basic relationship
$\mathcal{I}_k (h_L) \mathcal{I}_l (h_L) \subset \mathcal{I}_{k + l} (h_L)$ (which is due to H\"{o}lder inequality)
for $k, l \in \mathbb{N}$.

\subsection{Subadditivity of generalized numerical Kodaira dimension}
The following result obtained by O. Fujino is mainly motivated by the study of the subadditivity of the logarithmic Kodaira dimension of algebraic varieties, which is denoted by Conjecture $\overline{C}_{n,m}$ (see its precise form in e.g. Conjecture 1.2.2 in \cite{fuji-iitaka}). Note that Conjecture $C_{n,m}$ is a special case of Conjecture $\overline{C}_{n,m}$.

\begin{theorem}
  \label{thm1.1}{\tmem{(Theorem 1.3 in {\cite{fujino1}} or Theorem 2.1 in
  {\cite{fujino2}})}} Let $f : X \rightarrow Y$ be an algebraic
  fiber space between two projective manifolds $X$ and $Y$. Let $D_X$
  {\tmem{(resp. $D_Y$)}} be a simple normal crossing divisor (see Definition 9.1.7 in \cite{Larbook}) on $X$
  {\tmem{(resp. $Y$)}}. Assume that $f^{\ast} D_Y \subset D_X$, then
  \begin{equation}
    \kappa_{\sigma} (X, K_X + D_X) \geq \kappa_{\sigma} (F, K_F + D_F) +
    \kappa (Y, K_Y + D_Y) \label{spc}
  \end{equation}
  and
  \begin{equation}\label{spck}
    \kappa_{\sigma} (X, K_X + D_X) \geq \kappa  (F, K_F + D_F) +
    \kappa_{\sigma} (Y, K_Y + D_Y),
  \end{equation}
  where $F$ is a sufficiently general fiber of $f : X \rightarrow Y$.
\end{theorem}

Thus Conjecture $\overline{C}_{n,m}$ follows immediately from the generalized abundance conjecture (cf. Conjecture 2.10 in \cite{fujino1}), which implies that
$$  \kappa(X, K_X+D_X)=\kappa_{\sigma}(X,K_X +D_X)  $$ whenever $D_X$ is a simple normal crossing divisor on the projective manifold $X$.

Our first goal will be to generalize the above theorem with multiplier ideal sheaves involved. To this end, let us recall that
the key ingredient of the original proof of Theorem \ref{thm1.1} gives a
generically globally generated proposition of direct image sheaves (cf.
Corollary 3.7 in {\cite{fujino1}} or Theorem 3.35 of Chapter V in
{\cite{nakayama}}). His proof of this property is based on Nakayama's theory
of $\omega$-sheaves, which is closely related to Viehweg's covering trick and
weak positivity ({\cite{viehweg1}}, {\cite{viehweg2}}).

To convert this important (general) global generation property into a purely analytic statement (e.g. see Lemma \ref{lemma11} and Lemma \ref{lemma111}), we present
the following theorem, which is an application of Ohsawa-Takegoshi $L^2$
extension theorem on weakly pseudoconvex K\"{a}hler manifolds (cf. Theorem \ref{sdf} and Remark \ref{refer}). One might also compare it with Siu's uniform global
generation formula on pseudoeffective line bundles (cf. Chapter 6.E in {\cite{demailly-book-hep}}).

\begin{theorem}
  \label{ne}Let $f : X \rightarrow Y$ be an algebraic fiber space between two
  projective manifolds $X$ and $Y$, $L \rightarrow X$ a line bundle equipped
  with a singular hermitian metric $h_L$ whose curvature current is semi-positive.
  Then there exist two ample line bundles $G \rightarrow Y$ and $H \rightarrow
  X$ such that
  \begin{equation}
   \mathcal{F}_{1}:= \mathcal{O}_Y (G) \otimes f_{\ast} \mathcal{O }_X ((k K_{X / Y} + k L)
    \otimes \mathcal{I}_k (h_L))
  \end{equation}
 is sufficiently generically globally generated and
 \begin{equation}
    \mathcal{F}_{2}:= \mathcal{O}_Y (G) \otimes f_{\ast} \mathcal{O }_X (k K_{X / Y} + k L +
    H)
  \end{equation}
  is generically globally generated for any integer $k \geq 1$.
\end{theorem}
Here by saying a coherent sheaf $\mathcal{F}$ is ``\emph{generically globally generated}" (resp. ``\emph{sufficiently generically globally generated}"), we mean $\mathcal{F}$ globally generates its stalks on a Zariski open dense subset (resp. on a subset whose complement has measure zero).

Theorem \ref{ne} will be obtained through two extension statements: Lemma
\ref{lemma11} with multiplier ideal sheaves and Lemma \ref{lemma111} without multiplier ideal sheaves. As an immediate
application of Theorem \ref{ne} (or more precisely, Lemma \ref{lemma11}), we obtain an analogue of Theorem
\ref{thm1.1} for the subadditivity of generalized numerical Kodaira
dimensions.

\begin{theorem}[\textbf{Main Theorem 1}]
  \label{mainthm} Let $f : X \rightarrow Y$ be an algebraic
  fiber space between two projective manifolds $X$ and $Y$. Let L be a
  $\mathbb{Q}$-line bundle on $X$ equipped with a singular metric $h_L$
  whose curvature current is semi-positive. Then inequality
  \begin{equation}
    \kappa_{\sigma} (X, K_X + L, h_L) \geq \kappa_{\sigma} (F, K_F + L_F, h_L
    |_F) + \kappa (Y, K_Y) \label{112}
  \end{equation}
  and
  \begin{equation}
     \kappa_{\sigma} (X, K_X + L, h_L) \geq \kappa (F, K_F + L_F, h_L
    |_F) + \kappa _{\sigma}(Y, K_Y) \label{112k}
  \end{equation}
  holds, where $F$ is a sufficiently general fiber of $f : X \rightarrow Y$.
\end{theorem}

By ``{\tmem{sufficiently general}}'', we mean this inequality holds for $F =
X_y$ when $y$ varies outside a measure zero subset of $Y$. In terms of ({\ref{112k}), we will see from its proof that the left hand side $\kappa_{\sigma}$ can be optimized by a smaller number, which is defined to be the generalized numerical Kodaira dimension along the horizontal direction, i.e.
\begin{equation}\label{jiangluo}
\kappa_{\sigma,f,\text{hor}}:= \max_{E : \text{a divisor of $Y$}} \{\kappa_{\sigma} (X, K_X + L, h_L ;
   f^{\ast}E)\}\leq\kappa_{\sigma}.
\end{equation}

Let us remark that $\kappa_{\sigma} (F, K_F + L |_F, h_L |_F) \nobracket
\nobracket$ is invariant for sufficiently general $F.$
Indeed, for any $k, l \in \mathbb{N}$, there exists Zariski open
dense subsets $Y_{k, l} \subset Y$ such that $h^0 (X_y, k k_0 (X_y +
L|_{X_y}) \otimes l A \otimes \mathcal{I}_{k k_0} (h_L) |_{X_y})$ is invariant
when $y \in Y_{k, l}$ (e.g. Theorem 10.7 in \cite{bertin}). As $\mathcal{I}_{k k_0} (h_L) |_{X_y} = \mathcal{I}_{k
k_0} (h_L |_{X_y})$ holds for $y \in \Sigma_k$, where $Y \backslash \Sigma_k$ is a
subset with zero measure, it follows that $\kappa_{\sigma} (F, K_F + L_F, h_L
|_F)$ is constant on $\bigcap_{k, l = 1}^{\infty} (Y_{k, l}  \bigcap
\Sigma_k)$. Note also by almost the same argument, $\kappa (F, K_F + L|_F, h_L {|_F} )$ does not vary for almost all $F$.
One can refer to Remark 1.1 in {\cite{zhou-zhu1}} or Section \ref{sec2.1} for
further explanations of this result.

At the end of the proof of Theorem \ref{mainthm}, we will
provide an analytic approach towards
Theorem \ref{thm1.1} (cf. Remark \ref{alter}), which can be respectively viewed as an application of Theorem \ref{ne} (or more precisely, Lemma \ref{lemma111}).

\subsection{Subadditivity of generalized Kodaira dimension}

The second goal of this paper is to give another interpretation of the
subadditivity of generalized Kodaira dimension, which is mainly motivated by a
recent result of Zhou and Zhu.

\begin{theorem}
  {\tmem{\label{01.1}(Theorem 1.1 in {\cite{zhou-zhu1}})}} Let $f :
  X \rightarrow Y$ be an analytic fiber space between a compact K{\"a}hler
  manifold $X$ and a connected compact complex manifold $Y$. Let $L
  \rightarrow X$ be a  $\mathbb{Q}$-line bundle equipped with a singular metric $h_L$
  with semi-positive curvature current. Assume that $Y$ has general type, then
  inequality
  \begin{equation}\label{xunzhao}
    \kappa (X, K_X + L, h_L) \geq \kappa (F, K_F + L |_F, h_L |_F) + \dim Y
    \nobracket \nobracket
  \end{equation}
  holds, where $F$ denotes the sufficiently general fiber of $f$.
\end{theorem}

\begin{remark}\label{ams}
As a slight improvement to Theorem \ref{01.1}, we can show that (\ref{xunzhao}) is indeed an equality without any singularity restrictions on $h_L$ (cf. Proposition \ref{prop4.3}).
\end{remark}

Note that in \cite{zhou-zhu1} $\kappa (X, K_X + L, h_L)$ is defined (cf. \ref{pkj}) to be the largest order $m\in\mathbb{N}\bigcup\{-\infty\}$, where some subsequence of $\{h^0(X,k(K_X+L)\otimes\mathcal{I}_{k}(h_L))\}_{k=1}^{\infty}$ grows as  $O(k^m), k\rightarrow \infty$. 
 Essentially, Zhou and Zhu's proof of Theorem \ref{01.1} is mainly based on
two ingredients: Demailly's $L^2$ extension theory on weakly pseudoconvex
manifolds ({\cite{dem15}}) and their results on positivity of relative
$k$-Bergman metrics for K{\"a}hler fibrations ({\cite{zhouzhu}}).

However, we will try {\tmstrong{NOT}} to adopt the definition of generalized Kodaira dimension along this direction. Instead, we define $\tilde{\kappa} (X, K_X + L, h_L)$
(cf. (\ref{jkp}))
to be the maximal dimension of the image of the meromorphic Kodaira map
corresponding to the linear system
\begin{equation} \label{coplt}
  |H^0 (X, k k_0 (K_X + L) \otimes \mathcal{I}_{k k_0} (h_L)) |
\end{equation}
(cf. \ref{jkp}) as $k$ varies among all integers. We will show that both $\kappa (X, K_X + L, h_L)$ and  $\tilde{\kappa} (X, K_X + L, h_L)$ coincide with the dimension of an appropriate Newton-Okounkov body (cf. Proposition \ref{eqji}), even when $X$ is possibly non-algebraic, which will be crucial for the improvement of Theorem \ref{01.1} as mentioned in Remark \ref{ams}. Note in the usual setting (i.e. set $(L,h_L)$ to be the trivial line bundle equipped with the complex Euclidean metric), the fact that $\kappa(X,K_X) = \tilde{\kappa} (X,K_X)$ has already been proved in \cite{mori1}; the fact that both $\kappa(X,K_X)$ and $\tilde{\kappa}(X,K_X)$ coincide with an appropriate Newton-Okounkov body in case $X$ is algebraic has been proved in (Theorem 3.3 and 3.4 of) \cite{kk12}.

Before turning to the second main results, let us talk about the singularities
of $h_L$. Recall that $h_L$ is said to be of {\tmstrong{analytic singularities}}, if the local
weight function of $h_L$ has the form $c \cdot\log (|F_1 |^2 + \cdots + |F_N |^2) + O
(1)$, where the $F_j$ are all local holomorphic functions and $c>0$.

Now let us first fix the notation $S_k=H^0(X, k k_0(K_X+ L)\otimes \mathcal{I}_{k k_0}(h_L))$. Then the second main result of this paper, which guarantees the existence of generalized Iitaka fibration when $h_L$ has analytic singularities, can be stated as follows.

\begin{theorem}[=Theorem \ref{below}, \textbf{Main Theorem 2}]\label{newbo}

  \label{bow}Let $X$ be a compact complex manifold and $L \rightarrow X$ be
  a pseudoeffective line bundle equipped with a singular metric $h_L$ whose
  curvature current is semi-positive. Assume that $h_L$ has analytic
  singularities and $\kappa (X, K_X + L, h_L) \geq 0$. Let $f^{\flat} :
  X^{\flat} \rightarrow Y^{\flat}$ be a representative of the Kodaira meromorphic mapping
  \begin{equation}
    f (= \Phi_{|S_k |}) : X \dashrightarrow \tmop{Im} \Phi_{|S_k |}
    \backassign Y
  \end{equation}
  associated to $S_k$ for any $k \in N (X, K_X + L, h_L)$ sufficiently large. Then $f^{\flat} : X^{\flat} \rightarrow
  Y^{\flat}$ is an analytic fiber space, $\kappa(X, K_X + L, h_L) =
  \dim Y^{\flat}$ and $\kappa (F, K_F + L_F, h_L |_F) = 0$, where
  $F $ is the very general (which means, $F=X_y$ when $y$ varies outside a countable union of analytic
  subsets of $Y$) fiber of $f^{\flat}$.

\end{theorem}

\begin{remark}

Theorem \ref{below} generalizes a classical result on Iitaka fibration (cf. (i) of Theorem 1.11 in \cite{mori1}) with multiplier ideal sheaves involved.

\end{remark}

The key to the proof of Theorem \ref{newbo} is using the fact that $\kappa (X, K_X + L, h_L)=\tilde{\kappa}(X, K_X + L, h_L)$, where $\tilde{\kappa}(X, K_X + L, h_L)$ is defined by (\ref{coplt}). As an immediate application of Theorem \ref{newbo}, we are able to obtain a distinct proof of Theorem \ref{01.1} in the case when $h_L$ has analytic singularities.

\begin{corollary}[a special case of Theorem \ref{01.1}]
  \label{dio} Let $f : X
  \rightarrow Y$ be an analytic fiber space between a compact K{\"a}hler
  manifold $X$ and a projective manifold $Y$. Let $L \rightarrow X$ be a  $\mathbb{Q}$-line bundle equipped with a singular metric $h_L$ with semi-positive
  curvature current. Assume moreover that $h_L$ has
  analytic singularities and $Y$ has
  general type, then
  \begin{equation}
    {\kappa} (X, K_X + L, h_L) = {\kappa} (F, K_F + L |_F, h_L
    |_F) + \dim Y \nobracket \nobracket
  \end{equation}
  holds, where $F$ denotes the sufficiently general fiber of $f$.
\end{corollary}

\begin{remark}
  In the case when $(X, \Delta)$ is a \emph{compact K{\"a}hler} klt pair with a $\mathbb{Q}$-effective divisor
  $\Delta$, $Y$ is of general type and set $(L, h_L) = (\mathcal{O}_X (\Delta),
  h_{\Delta})$, it follows immediately from Corollary \ref{dio} that
  \begin{equation}
    \kappa (X, K_X + L) = \kappa (F, K_F + L_F) + \dim Y. \label{jop}
  \end{equation}
  Such addition formula (\ref{jop}) has also been proved in {\cite{campana}} and
  {\cite{nakayama}} when $(X, \Delta)$ is a \emph{projective} log
  canonical pair (see also Theorem 1.7 in {\cite{fujino-notes}}).
\end{remark}

The proof of Corollary \ref{dio} will be different from that of Theorem \ref{01.1}. It is based on a quantitative analysis of
the classical Kodaira meromorphic map (cf. Theorem \ref{newbo}) and some recent developments
of analytic tools, which mainly consists of $L^2$ extension theory and the
positivity of relative $k$-Bergman metrics.

\quad

This paper is organized as follows. In section \ref{sect2}, we give some
preliminaries for the proofs of main theorems. The most important results will be the extension results: Lemma \ref{lemma11} and Lemma \ref{lemma111}. In section \ref{sect3}, we will complete
the proofs of Theorem \ref{ne} and Theorem \ref{mainthm}. We also provide an analytic approach to Theorem \ref{thm1.1}. In section \ref{sect5}, we first recall some
basic results about Iitaka $D$-dimension and Okounkov body associated to graded algebra of almost integral type. Then we discuss their applications to generalized
Kodaira dimensions. Consequently, we give the proof of Theorem \ref{newbo} in the end. In section \ref{sect4}, we give another proof of Corollary \ref{dio} in the special case when $h_L$ has analytic singularities, which will be different from that of \cite{zhou-zhu1}.

\section{Preliminaries}\label{sect2}

\subsection{Relative $k$-Bergman metrics for K{\"a}hler
fibrations}\label{sec2.1}

Let $f : X \rightarrow Y$ be a surjective holomorphic map between an
$n$-dimensional compact K{\"a}hler manifold $X$ and an $m$-dimensional compact
connected complex manifold $Y$. Let $L$ be a pseudoeffective line bundle over $X$
equipped with a (semi-)positively curved singular hermitian metric $h_L = e^{-
\varphi_L},$ where $\varphi_L \in L^1_{\tmop{loc}}$ is a plurisubharmonic local weight function.
For $k \in \mathbb{N}$, the stalk of the $k$-multiplier ideal sheaf $\mathcal{I}_k (h_L)
\subset \mathcal{O}_X$ at $x\in X$ is defined as follows:
\begin{equation}
  \mathcal{I}_k (h_L)_x \assign \{F \in \mathcal{O}_{X, x} ; |F|^{2 / k} e^{-
  \varphi_L} \in L^1_{\tmop{loc}} \tmop{near} x\} . \label{ddjo}
\end{equation}
When $k = 1$, It is well-known from Nadel's theorem that $\mathcal{I} (h_L)=\mathcal{I}_{1} (h_L)$ is a
coherent analytic sheaf. For general $k$, the coherence of the $k$-multiplier ideal sheaf follows from the strong
openness property of the multiplier ideal sheaves ({\cite{guan-zhou15soc}}) and
Demailly's equisingular approximation theorem ({\cite{dps}}) for
plurisubharmonic functions (cf. {\cite{Siu04}, \cite{cao}}).

Let $K_{X / Y} \assign K_X - f^{\ast} K_Y$ be the relative canonical bundle.
We set
\[ Y_0 \assign \{ y \in Y | \tmop{rank} d f (x) = m, \forall x\in f^{-1}(y) \} \nobracket \]
to be the regular values of $f$, $Y_h \assign \{ y \in Y_0 : h_L |_{X_y}
\nequiv + \infty \} \nobracket$ and
\begin{equation}
  Y_{k, \tmop{ext}} \assign \{ y \in Y_0  | \tmop{rank} f_{\ast} (k K_{X / Y}
  + k L) = h^0 (X_y, k K_{X_y} + k L|_{X_y}) \} . \nobracket
\end{equation}
and respectively $X_{k, \tmop{ext}} \assign f^{- 1} (Y_{k, \tmop{ext}})$. It
is clear that both $Y_0$ and $Y_{k, \tmop{ext}}$ are Zariski open dense subsets of
$Y$. Assume that
\[ H^0 (X_y, (k K_{X_y} + k L|_{X_y}) \otimes \mathcal{I}_k (h_L)) \neq 0 \]
holds for some $y \in Y_{k, \tmop{ext}}  \bigcap Y_h$. Now we can assign the
$k$-Bergman metric, denoted by $B^{- 1}_{k, y}$, where
\[ B_{k, y} (x) \assign \sup \{u (x) \otimes \overline{u} (x) ; u \in H^0
   (X_y, k K_{X_y} + k L|_{X_y}) \infixand \int_{X_y} |u|^{2 / k} h_L \leq
   1\}, \]
to each fiber $X_y$ for $y \in Y_{k, \tmop{ext}}$. Gluing them together we may
endow $(k K_{X / Y} + k L) |_{X_{k, \tmop{ext}}}$ with a metric denoted by
$B^{- 1}_{k, X / Y}$, which is known as the relative $k$-Bergman metric. It
has been proved that

\begin{theorem}
  \label{zhouzhu}{\tmem{(Theorem 1.5 in {\cite{zhouzhu}})}} $B^{- 1}_{k, X /
  Y}$ has semi-positive curvature current on $X_{k, \tmop{ext}}$. Furthermore,
  it can be extended across $X \backslash X_{k, \tmop{ext}}$ as a
  new metric with semi-positive curvature current.
\end{theorem}

\begin{remark}
  When $f$ is projective, this result was initially obtained in {\cite{PT}}
  (see also {\cite{bp08}}). When $k = 1$ and
  $f$ is a proper fibration between a K{\"a}hler manifold $X$ and a complex
  manifold $Y$ ($X$ not necessarily compact), this result was obtained in
  {\cite{zhouzhujdg}} by using the so-called Guan-Zhou method in
  {\cite{guan-zhou15optimal}}. For general $k \geq 2$, a crucial step in \cite{zhouzhu} to
  establish Theorem \ref{zhouzhu} is overcoming the extension difficulties
  which were pointed out in Remark 4.2.4 in {\cite{PT}}.
\end{remark}

Now consider the set $\tilde{Y}_{k, h_L, \tmop{ext}} \assign \{y \in Y_0 ; h^0
(X_y, (k K_{X_y} + k L|_{X_y}) \otimes \mathcal{I}_k (h_L) |_{X_y}) =
\tmop{rank} f_{\ast} (k K_{X / Y} + k L) \otimes \mathcal{I}_k (h_L)\}$ and
respectively $$Y_{k, h_L, \tmop{ext}} \assign \{y \in Y_{k, \tmop{ext}}
\bigcap \tilde{Y}_{k, h_L, \tmop{ext}} ; \mathcal{I}_k (h_L) |_{X_y} =
\mathcal{I}_k (h_L |_{X_y}) \} \nobracket \nobracket.$$ In {\cite{dwzz}}, the
authors has obtained another new proof of the positivity of the relative
$k$-Bergman metric on $f^{- 1} (\tilde{Y}_{k, h_L, \tmop{ext}})$. According to Fubini's theorem and the basic extension property of
$k$-multiplier ideal sheaves (cf. Theorem 2.4 in {\cite{zhou-zhu1}} or
Proposition 0.2 in {\cite{bp2}}), one may find that $\mathcal{I}_k (h_L) |_{X_y} = \mathcal{I}_k
(h_L |_{X_y}) \nobracket \nobracket$ for almost all $y \in Y$ and $Y
\backslash Y_{k, h_L, \tmop{ext}}$ is of measure zero. As a consequence,  for any $y \in \bigcap^{\infty}_{ k = 1} Y_{k,
h_L, \tmop{ext}}$, $h^0
(X_y, (k K_{X_y} + k L|_{X_y}) \otimes \mathcal{I}_k (h_L |_{X_y}))$ is
constant for any $k \in \mathbb{N}$, which explains why $\kappa (F, K_F + L|_F, h_L {|_F} )$
does not vary for almost all $F$ in Theorem \ref{01.1}.

The following proposition will be useful for clarifying the relationship between
Nadel's multiplier ideal sheaves and $k$-multiplier ideal sheaves:

\begin{proposition}
  \label{relation}Let $L$ be a line bundle, $k \in \mathbb{N}$ and $B^{- 1}_{k, X / Y}$ be the natural extension of the
  relative Bergman metric on $k (K_{X / Y} + L)$ according to Theorem \ref{zhouzhu}, then
  \begin{enumeratealpha}
    \item $\Gamma (X_y, k (K_{X_y} + L|_{X_y}) \otimes \mathcal{I}_k (h_L
    |_{X_y})) \subset \Gamma (X_y, k (K_{X_y} + L|_{X_y}) \otimes \mathcal{}
    \mathcal{I} ((B_{k, X / Y}^{- \frac{k - 1}{k}}  \cdot h_L )|_{X_y}))$
    for general $X_y$;

    \item Let $G$ be a line bundle on $Y$. Then $\Gamma (X, k (K_{X / Y} + L)
    \otimes f^{\ast} G \otimes \mathcal{I} (B_{k, X / Y}^{- \frac{k - 1}{k}}
    \cdot h_L)) \subset \Gamma (X, k (K_{X / Y} + L) \otimes f^{\ast} G
    \otimes \mathcal{I }_k (h_L)) .$
  \end{enumeratealpha}
\end{proposition}

\begin{proof}

  \begin{enumeratealpha}
    \item Note $B^{- 1}_{k, X / Y} |_{X_y}$ for $y \in Y_{k, \tmop{ext}}$ is
    nothing but the $k$-Bergman kernel metric on $X_y$, therefore, for any $s
    \in \Gamma (X_y, k (K_{X_y} + L|_{X_{y}}) \otimes \mathcal{I}_k (h_L |_{X_y}))$ we
    get
    \begin{equation}
      \int_{X_y} |s|^2 B_{k, X / Y}^{- \frac{k - 1}{k}} \cdot h_L \leq
      \int_{X_y} |s|^2 \cdot (\frac{\|s\|_k}{|s|})^{2 \frac{k - 1}{k}} \cdot
      h_L = \|s\|_k^2, \label{24}
    \end{equation}
    where
    \begin{equation}
      \|s\|_k  : = (\int_{X_y} |s|^{2 / k} h_L)^{k / 2} .
    \end{equation}
    Our claim follows by (\ref{24}).

    \item We first denote that $\tilde{h} \assign B_{k, X / Y}^{- \frac{k -
    1}{k}} \cdot h_L$. For any $\tilde{s} \in \Gamma (X, k (K_{X / Y} + L)
    \otimes f^{\ast} G \otimes \mathcal{I }_k (h_L))$, and any smooth metric
    $h_{K_Y}$ (resp. $h_G$) on $K_Y$ (resp. $G$), we have
    \[ \int_X | \tilde{s} |^{\frac{2}{k}} h_L  \cdot f^{\ast}
       (h_G^{\frac{1}{k}} h^{- 1}_{K_Y }) = \int_X \frac{| \tilde{s}
       |^{\frac{2}{k}} f^{\ast} (h_G h_{K_Y}^{- 1})^{\frac{1}{k}}}{B_{k, X /
       Y}^{\frac{k - 1}{k^2}} } \cdot \frac{B_{k, X / Y}^{\frac{k -
       1}{k^2}}}{{(f^{\ast}}  h_{K_Y})^{\frac{k - 1}{k}}} \cdot h_L \]

    \[ \text{(H\"{o}lder inequality)} \leq (\int_X \frac{|
       \tilde{s} |^2 h_L f^{\ast} (h_G h_{K_Y}^{- 1}) }{B_{k, X / Y}^{\frac{k
       - 1}{k }}})^{\frac{1}{k}} (\int_X \frac{B_{k, X / Y}^{\frac{1}{ k}}
       h_L}{f^{\ast} (h_{K_Y})})^{\frac{k - 1}{k}} \]
    \[ = (\int_X | \tilde{s} |^2 \tilde{h} f^{\ast} (h_G h_{K_Y}^{-
       1}))^{\frac{1}{k}} (\int_X \frac{B_{k, X / Y}^{\frac{1}{k }}
       h_L}{f^{\ast} (h_{K_Y})})^{\frac{k - 1}{k}} \]
       \begin{equation}
         \overset{\tmop{Fubini}}{=} \overset{}{} (\int_X | \tilde{s} |^2
         \tilde{h}  f^{\ast} (h_G h_{K_Y}^{- 1}))^{\frac{1}{k}} (\int_{y \in
         Y} \frac{1}{h_{K_Y}} \int_{X_y} B_{k, y}^{\frac{1}{k }} h_L)^{\frac{k
         - 1}{k}} . \label{223}
       \end{equation}
    To analyze the last integral, one just needs to use the fact that
    \begin{equation}
      \int_{X_y} B_{k, y}^{\frac{1}{k }} h_L \leq \tmop{rank} f_{\ast} ((k
      K_{X / Y} + L) \otimes \mathcal{I}_k (h_L) \mathcal{}) \label{224} <
      \infty
    \end{equation}
    holds for $y \in Y_{k, \tmop{ext}}$ (cf. Lemma 2.7 in {\cite{zhou-zhu1}}).
    Finally we verify our claim from inequality (\ref{223}).
  \end{enumeratealpha}
\end{proof}

\begin{remark}
  \label{suppl}When $L$ is a pseudoeffective $\mathbb{Q}$-line bundle which can be equipped with a (semi-)positively curved singular
  metric $h_L$, Proposition \ref{relation} is still true for any $k \in
  \mathbb{N}$ sufficiently divisible (so that $k L$ is a line bundle).
\end{remark}

\subsection{$L^2$ Extension theory on weakly pseudoconvex manifolds}

We will adopt the following Ohsawa-Takegoshi $L^2$ extension theorem on weakly pseudoconvex
K{\"a}hler manifolds:

\begin{theorem}
  \label{sdf}{\tmem{(Theorem 2.8 and Remark 2.9 (b) in {\cite{demailly16}}, or Theorem 2.1 in
  {\cite{zhou-zhu1}})}} Let $(X, \omega)$ be a weakly pseudoconvex K{\"a}hler
  manifold and $\psi$ be a quasi-psh function on $X$ with neat analytic
  singularities, {\tmem{i.e.}} $\psi$ can be locally written as:
  \[ \psi = c \log ( | F_1 |^2 + \cdots + | F_N |^2) + u, \nobracket
     \nobracket \nobracket \nobracket \]
  where each $F_j$ is holomorphic, $c \in \mathbb{R}_{\geq 0}$ and $u$ is
  smooth. Let S be the zero variety $V (\mathcal{I} (\psi))$. Assume that $\psi$
  has log canonical singularities, {\tmem{i.e.}}
  \[ \mathcal{I} ((1 - \varepsilon) \psi) |_S = \mathcal{O }_X |_S \nobracket
     \nobracket \]
  holds for every $\varepsilon > 0$. Denote by $S^0$ the regular set of $S$. Let $P$
  be a holomorphic line bundle over $X$ equipped with a singular hermitian
  metric $h_P$. Assume that there is an $\alpha > 0$ such that
  \begin{equation}
    \sqrt{- 1} \Theta_{h_P} + \xi \sqrt{- 1} \partial \overline{\partial} \psi
    \geq 0
  \end{equation}
  holds for all $\xi \in [1, 1 + \alpha]$. Then for every
  \[ f \in \Gamma (S^0, (K_X + P) |_{S^0}) \nobracket \]
  and
  \[ \int_{S^0} | f |_{h_{\omega} \otimes h_P}^2 d V_{X, \omega} [\psi] <
     \infty, \nobracket \nobracket \]
  there exists an $F \in \Gamma (X, K_X + P)$ such that $F |_S = f \nobracket$
  and
  \begin{equation}
    \int_X \frac{\rho (\alpha \psi)}{e^{\psi}} |F|^2_{h_{\omega} \otimes h_P}
    d V_{X, \omega} \leq \frac{34}{\alpha} \int_{S^0} |f |_{h_{\omega} \otimes
    h_P}^2 d V_{X, \omega} [\psi], \nobracket
  \end{equation}
  where $\rho : \mathbb{R \rightarrow \mathbb{R}}$ is defined by $\rho (t) =
  e^{- \frac{t}{2}}$ when $t > 0$, $\rho (t) = \frac{1}{1 + t^2}$ when $t \leq
  0$.
\end{theorem}

\begin{remark}\label{refer}
  Here the Ohsawa measure $d V_{X, \omega} [\psi]$ on $S^0$ is defined by
  \begin{equation}
    \int_{S^0}  g d V_{X, \omega} [\psi] \assign \limsup_{t \rightarrow - \infty}
    \int_{\{ x : t < \psi < t + 1 \}} \tilde{g} e^{- \psi} d V_{X, \omega},
  \end{equation}
  where $\tilde{g}$ is any smooth extension of a non-negative continuous
  function $g \in \mathcal{C}  (S^0)$ with compact support. Originally, the measure $d V_{X, \omega} [\psi]$ was claimed to be
  well-defined on $S^0$ by Proposition 4.5 in
  {\cite{demailly16}}. 
  
  However, as the latest paper \cite{dkim} suggests, this claim was made under the implicit ``unique lc place" assumption, which says that each irreducible component of the non-klt locus $S$ is dominated by a unique divisor with discrepancy equals to $-1$. It can be checked easily that when  $\psi$ is locally of the form $\psi = k \cdot \log(|z_1|^2 + ... + |z_k|^2) + \mathcal{C}^{\infty}$ (which is also the setting for later use), where $(z_1 = ... = z_k = 0)$ are local defining equations for $S$ (thus $S$ is non-singular and $k$ is the codimension of $S$), the so-called ``unique lc place" assumption is automatically satisfied and the Ohsawa measure $d V_{X, \omega} [\psi]$  has smooth positive density with respect to the measure $d V_{S,\omega|_{S}}$ (cf. (2.6) in \cite{demailly16}). In fact, this special log canonical singularity assumption on $\psi$ is the setting where the Ohsawa measure first appeared (cf. \cite{ohsawa4}).
  
  One can
  also refer to {\cite{zhouzhupacific}} for an optimal version of the above
  theorem.
\end{remark}

\begin{lemma}
  Let \label{lemma11}$f : X \rightarrow Y$ be a surjective holomorphic mapping
  between a compact K{\"a}hler manifold $X$ and a compact complex manifold
  $Y$. Let $L \rightarrow X$ be a pseudoeffective $\mathbb{Q}$-line bundle
  equipped with a singular metric $h_L$ whose curvature current is semi-positive.
  Let $k_0$ be the least integer so that $k_0 L$ is a line bundle and let $G$
  be an ample line bundle over $Y$. Let $\psi \leq 0$ be a quasi-psh function
  on $Y$, which is smooth outside a very general point $y \in
  \bigcap_{k = 1}^{\infty} Y_{k k_0, \tmop{ext}}$ and assume that
  \[ \psi (z) - \log | z - y  |^{2 m} \nobracket \nobracket \]
  is smooth near $y$ {\tmem{(such $\psi$ always exists! cf. (\ref{podq}))}}.
  Moreover, assume that there exists a hermitian metric ${h_{G - K_Y}}  $on $G
  - K_Y$ such that:
  \begin{equation}
    \sqrt{- 1} \Theta_{h_{G - K_Y}} (G - K_Y) + \xi \sqrt{- 1} \partial
    \overline{\partial} \psi \geq 0 \label{dddd}
  \end{equation}
  for all $\xi \in [1, 1 + \alpha]$ and $\alpha > 0$. Then for any integers
  $k$ and
  \[ s  \in \Gamma \left( X_{y }, k k_0 \left( K_{X_{y }} + L|_{X_y} \right)
     \otimes \mathcal{\mathcal{I} }_{k k_0} (h_L |_{X_y}) \right) \]
  and $y \in \bigcap_{k = 1}^{\infty} Y_{k k_0, \tmop{ext}}$, there exists a section
  \[ \tilde{s} \in \Gamma (X, (k k_0 (K_{X / Y} + L) + f^{\ast} G) \otimes
     \mathcal{I }_{k k_0} (h_L)) \]
  such that $\tilde{s} |_{X_y} = s$.
\end{lemma}

\begin{proof}
  This is proved by a similar argument of Theorem 1.2 in {\cite{zhou-zhu1}}.
  Let
  \begin{equation}
    P \assign (k k_0 - 1) (K_{X / Y} + L) + L + f^{\ast} (G - K_Y) \label{jjj}
  \end{equation}
  be a line bundle endowed with a singular hermitian metric:
  \begin{equation}
    h_P \assign B_{k k_0, X / Y}^{- (k k_0 - 1) / k k_0} \cdot f^{\ast} h_{G -
    K_Y} \cdot h_L . \label{das}
  \end{equation}
  According to (\ref{dddd}), (\ref{das}) and Theorem \ref{zhouzhu}, we may
  compute its curvature by:
  \[ \sqrt{- 1} \Theta_{h_P} + \xi \sqrt{- 1} \partial \overline{\partial}
    f^{*} \psi = (kk_{0} - 1) \sqrt{- 1} \Theta_{ B_{k k_0, X / Y}^{-1 / k
     k_0}}(K_{X / Y} + L) + \sqrt{- 1} \Theta_{h_L} (L) + \]
  \begin{equation}
    f^{\ast}(\xi \sqrt{- 1}  \partial \overline{\partial} \psi +
    \sqrt{- 1} \Theta_{{h_{G - K_Y}} } (G - K_Y)) \geq 0.
  \end{equation}
  Note $V (\mathcal{J} (f^{\ast} \psi)) = X_y$ and $f^{\ast} \psi$ has log
  canonical singularities (which also satisfies the ``unique lc place" assumption), applying Theorem
  \ref{sdf} and Remark \ref{refer} we see that
  $$\Gamma(X,(K_X + P)\otimes \mathcal{I}(h_P))= \Gamma (X, (k k_0 (K_{X / Y} + L) + f^{\ast} G) \otimes \mathcal{I} (B_{k
     k_0, X / Y}^{- \frac{k k_0 - 1}{k k_0}} \cdot h_L)) \rightarrow$$
     $$
      \Gamma(X_y,(K_{X} + P)|_{X_y}\otimes \mathcal{I}(h_P |_{X_{y}}))=\Gamma
     (X_y, k k_0 (K_{X_y} + L|_{X_{y}}) \otimes \mathcal{} \mathcal{I} (B_{k k_0, X /
     Y}^{- \frac{k k_0 - 1}{k k_0}} |_{X_y} \cdot h_L |_{X_y})) $$
  is surjective for every $y \in Y_0$. Combining relation $a)$ and $b)$ in
  Proposition \ref{relation} (see also Remark \ref{suppl}) we immediately get
  the desired extension section $\tilde{s}$ when $y \in \bigcap_{k =
  1}^{\infty} Y_{k k_0, \tmop{ext}}$.
\end{proof}

\begin{remark}
  One may refer to Theorem 2.11 in {\cite{dengya}} for a qualitative
  formulation of this result under the assumption that $f$ is projective. Our
  improvement consists of two aspects. One is that it gives a more precise
  version of a quantitative estimate of the restriction map, the other is that
  we only assume $f : X \rightarrow Y$ to be a K{\"a}hler fiber space with the
  help of Theorem \ref{zhouzhu}. Compared with Theorem 1.2 in
  {\cite{zhou-zhu1}}, we do not assume the positivity of the canonical bundle
  $K_Y$ to control the negative part of the 
  closed $(1, 1)$-current
  $f^{\ast} \sqrt{- 1} \partial \overline{\partial} \psi .$
\end{remark}

We can also formulate a slight modified version of Lemma
\ref{lemma11}. The main difference is that the inequality does not involve any
multiplier ideal sheaves.

\begin{lemma}  \label{lemma111}
  Let $f : X \rightarrow Y$ be an algebraic
  fiber space between two projective manifolds $X$ and $Y$.
  Let $k \in \mathbb{N}$. Let $L=\mathcal{O}_X (D)$ be a line bundle over $X$, where $D$ is a simple normal crossing divisor on $X$. Let $G$ be an ample line bundle over
  $Y$. Let $H$ be an ample line bundle over $X$. Let $y$ be a very
  general point in $Y$ and
  $\psi \leq 0$ be a quasi-psh function on $Y$ which is smooth outside $y$. Assume that
  $$ \psi (z) - \log | z - y  |^{2 m} \nobracket \nobracket $$
  is smooth near $y$. Moreover, assume that there exists a \ hermitian metric
  ${h_{G - K_Y}}  $on $G - K_Y$ such that:
  \begin{equation}
    \sqrt{- 1} \Theta_{h_{G - K_Y}} (G - K_Y) + \xi \sqrt{- 1} \partial
    \overline{\partial} \psi \geq 0 \label{ddd}
  \end{equation}
  for all $\xi \in [1, 1 + \alpha]$ and $\alpha > 0$. Then for any $k\in\mathbb{N}$ and
  \[ s  \in \Gamma ( X_{y }, k ( K_{X_{y }} + L|_{X_y} ) +
     H|_{X_y}), \]
  there exists a section
  \[ \tilde{s} \in \Gamma (X, (k (K_{X / Y} + L) + H + f^{\ast} (G))
     \nobracket \]
  such that $\tilde{s} |_{X_y} = s$.
\end{lemma}

\begin{proof}

 Since $D$ is a simple normal crossing divisor, for each $k\in\mathbb{N}$, the big $\mathbb{Q}$-line bundle $L_k := L+\frac{1}{k} H$ can be equipped with a singular metric $h_k$ with analytic singularities, whose curvature current is semi-positive and multiplier ideal sheaf $\mathcal{I}(h_k)$ is trivial. Thus $\mathcal{I}(h_k |_{X_y})$ is also trivial for general $y\in Y$. Hence $||s||_k^{2/k}=\int_{X_y} |s|^{2/k}h_k <\infty$. Let $B_{k,X/Y,L_k}^{-1}$ denote the (extended) relative $k$-Bergman kernel metric on $k(K_{X/Y}+L_k)$, with respect to the algebraic fiber space $f:X\rightarrow Y$ and the metric $h_k$ on $L_k$. Fix the integer $k \in \mathbb{N}$.
  Then for general $y\in Y$ and $s \in \Gamma ( X_{y }, k ( K_{X_{y }} + L|_{X_y} )
  + H|_{X_y} )=\Gamma(X_y, k(K_{X_y}+L_k |_{X_{y}}))$, the following inequality
  \begin{equation}\label{jsp}
 (B^{-
    \frac{1}{k}}_{k, X / Y, L_k}) |_{X_y}= B^{-
    \frac{1}{k}}_{k, y, h_k}\leq
    |\frac{s}{||s||_k}|^{- \frac{2}{k}}
  \end{equation}
  holds, where $B_{k, y, h_k}$ stands for the $k$-Bergman kernel of $k(K_{X_y}+L_k |_{X_y})$ with respect to the metric $h_k |_{X_y}$.
  Let
    \begin{equation}
    P \assign (k - 1) (K_{X / Y} + L_k) + L_k + f^{\ast} (G - K_Y) \label{jjj}
  \end{equation}
  be a line bundle over $X$.
  Now we define a metric on $P$ as
  \begin{equation}
    h_{P } \assign B_{k, X / Y, L_k}^{- (k -
    1)/k}  \cdot h_{k} \cdot f^{\ast} h_{G - K_Y} \label{29}
  \end{equation}
 whose curvature current is semi-positive, according to (\ref{ddd}) and Theorem \ref{zhouzhu}.
 Then for general $y \in Y$,
  \begin{equation}
   C_y \cdot \int_{X_y}|s|^{2} h_P = \int_{X_y} |s|^2 B_{k, X / Y, L_k}^{- (k - 1)/k} \cdot h_k  \leq
    ||s||_{k}^{\frac{2(k-1)}{k}}\cdot\int_{X_y} |s|^{\frac{2}{k}} \cdot h_k =||s||^{2}_{k}
    < \infty
  \end{equation}
  holds thanks to (\ref{jsp}), where $C_y >0$ is a constant depending only on $y$. Therefore, applying Theorem \ref{sdf} and Remark \ref{refer} again we get a section
  \[ \tilde{s} \in \Gamma (X, K_X + P ) = \Gamma (X, (k (K_{X / Y} + L) + H
     + f^{\ast} (G)) \nobracket \]
  such that $\tilde{s} |_{X_y} = s$. As $k$ varies among all integers, the desired extension property holds for very general $y\in Y$.
\end{proof}

\section{Proof of Theorem \ref{ne} and Theorem \ref{mainthm}} \label{sect3}

Before giving the proof, let us give some simple but useful lemmas.

\begin{lemma}\label{simple}
  Let $X$ be a compact complex manifold and $(L,h_L)$ be a pseudoeffective line bundle over $X$. Let $a \in \mathbb{N}$ and $\kappa_{\sigma, a, A} (X, K_X + L,
  h_L)$ be the following number
  \begin{equation}
    \max \{\nu \in \mathbb{Z} ; \limsup_{k \rightarrow \infty} \frac{h^0 (k a
    (K_X + L) \otimes A \otimes \mathcal{I }_{a k} (h_L))}{k^{\nu}} > 0\}
  \end{equation}
  and $\kappa_{\sigma, a} \assign \max_{l \in \mathbb{N}} \{\kappa_{\sigma, a,
  l A} \}$. Then $\kappa_{\sigma, a} = \kappa_{\sigma}$.
\end{lemma}

\begin{proof}
  It suffices to prove $\kappa_{\sigma, a} \geq \kappa_{\sigma}$. By
  the definition of $\kappa_{\sigma}$ there exists $l_0$ and a sequence $\{k_{\nu} \}_{\nu=1}^{\infty}$ such that
  $N_{\nu} : = h^0 (k_{\nu} (K_X + L) \otimes l_0 A \otimes \mathcal{I
  }_{k_{\nu}} (h_L)) \geq C \cdot k_{\nu}^{\kappa_{\sigma}}$, then we obtain
  \begin{equation}
    h^0 (a k_{\nu} (K_X + L) \otimes a l_0 A \otimes \mathcal{I }_{a k_{\nu}}
    (h_L)) \geq N_{\nu} \geq C \cdot k_{\nu}^{\kappa_{\sigma}}
  \end{equation}
  for any $\nu \in \mathbb{N}$. Therefore, $\kappa_{\sigma, a} \geq
  \kappa_{\sigma, a, a l_0 A} \geq \kappa_{\sigma}$.
\end{proof}

\begin{lemma} \label{addti}
Let $f : X \rightarrow Y$ be an algebraic
  fiber space between two projective manifolds $X$ and $Y$.
 Let  $D_Y$ be an effective divisor on $Y$ and $\mathcal{E}$ be a coherent sheaf on $X$. Assume $\mathcal{F} =
  f_{\ast} \mathcal{E}$ is a sufficiently generically globally generated
  coherent sheaf on $Y$. Then
  \begin{equation}
    h^0 (X, \mathcal{E} \otimes f^{\ast} \mathcal{O}_Y (D_Y)) \geq h^0 (Y,
    \mathcal{O}_Y (D_Y)) \cdot \tmop{rk} \mathcal{F},
  \end{equation}
  where $\tmop{rk} \mathcal{F}$ is the rank of $\mathcal{F}$.
\end{lemma}

\begin{proof}
  Let us assume $q = h^0 (Y, \mathcal{O}_Y (D_Y)) \geq 1, r = \tmop{rk}
  \mathcal{F} \geq 1$ without loss of generality. According to projection
  formula, $H^0 (X, \mathcal{E} \otimes f^{\ast} \mathcal{O}_Y (D_Y))$=$H^0
  (Y, \mathcal{O}_Y (D_Y) \otimes \mathcal{F})$. It follows from (generalized) Grauert's direct image and semi-continuity theorem (cf. Theorem 10.6 and Theorem 10.7 in \cite{bertin}) that $\mathcal{F} |_{Y_0}$
  is locally free for some analytic Zariski open dense subset $Y_0$ contained in the regular values of $f$ and
  $\mathcal{F}_y \otimes k (y)  = H^0 (X_y, \mathcal{E} |_{X_y})$, $\tmop{rk}
  \mathcal{F} = h^0 (X_y, \mathcal{E}  |_{X_y})$ for any $y \in Y_0$. Since
  $\mathcal{F} = f_{\ast} \mathcal{E}$ is sufficiently generically globally
  generated, one may pick $y_0 \in Y_0$ and $r$ linearly independent sections
  $\sigma_1, \ldots, \sigma_r \in H^0 (Y, \mathcal{F})$ such that $\sigma_{1,
  y_0}, \ldots, \sigma_{r, y_0}$ generates the stalk $\mathcal{F}_{y_0}$. Note
  $\mathcal{F}$ is coherent, hence there is a small open neighborhood $\Omega
  \subset Y_0$ of $y_0$, such that $\sigma_{1, y }, \ldots, \sigma_{r, y }$
  generates the stalk $\mathcal{F}_{y }$ for any $y \in \Omega$. In
  particular, $\sigma_1 |_{X_y}, \ldots, \sigma_r |_{X_y}$ form a basis of
  $H^0 (X_y, \mathcal{E} |_{X_y})$ for $y \in \Omega$.

  Let $\xi_1, \ldots_, \xi_q$ be a basis of $H^0 (Y, \mathcal{O}_Y (D_Y))$.
  We claim that $\{\xi_i \sigma_j \}_{i = 1, \ldots q, j = 1, \ldots, r}$
  are linearly independent in $H^0 (Y, \mathcal{O}_Y (D_Y) \otimes \mathcal{F})
  = H^0 (X, \mathcal{E} \otimes f^{\ast} \mathcal{O}_Y (D_Y))$ and $h^0 (X,
  \mathcal{E} \otimes f^{\ast} \mathcal{O}_Y (D_Y)) \geq q r$.

  Actually, consider the following equations for $t_{i j} \in \mathbb{C}$:
  \begin{equation}
    \sum_{i = 1, \ldots q, j = 1, \ldots r} t_{i j} \xi_i \sigma_j = 0,
    \label{jjk}
  \end{equation}
  We will show $t_{i j} = 0$. Take $q$ general points $y_1, \ldots, y_q \in
  \Omega$ which will be determined later. Assume $\Omega$ also trivializes
  $\mathcal{O}_Y (D_Y)$ with a holomorphic frame $e$ and $\xi_i |_{\Omega} = \tilde{\xi}_i
  (y) \cdot e$ with $\tilde{\xi}_i \in \mathcal{O} (\Omega)$. When restricting
  (\ref{jjk}) to $X_{y_1}, \ldots X_{y_q}$, we find $q$ equations:
  \begin{equation}
    \sum_{j = 1}^r (\sum_{i = 1}^q t_{i j} \tilde{\xi}_i (y_k) \cdot f^{\ast}
    e) \sigma_j |_{X_{y_k}} = 0, k = 1, \ldots q
  \end{equation}
  Since $\sigma_1 |_{X_y}, \ldots, \sigma_r |_{X_{y_k}}$ are linearly
  independent,
  \begin{equation}
    \sum_{i = 1}^q t_{i j} \tilde{\xi}_i (y_k) = 0, j = 1, \ldots, r, k = 1,
    \ldots q.
  \end{equation}
  Fix $\tmop{any}$ $j$ now and consider the $q \times q$ matrix
  $[\tilde{\xi}_i (y_k)] $, if $[\tilde{\xi}_i (y_k)]$ is non-degenerate, then
  $t_{i j} = 0$ for all $i = 1, \ldots q$. To this end, we will pick $y_k \in
  \Omega$ as follows. \

  We will argue step by step. First choose $y_1$ such that $\tilde{\xi}_1
  (y_1) \neq 0$, this is possible since $\xi_1 \neq 0$. Consider the analytic
  subset $A_2 \assign \{y \in \Omega ; \widetilde{\xi_1} (y_1) \tilde{\xi}_2
  (y) - \widetilde{\xi_2} (y_1) \widetilde{\xi_1} (y) = 0\}$ of $\Omega$. It
  can be shown that $A_2 \neq \Omega$ since otherwise $\widetilde{\xi_1} (y_1)
  \xi_2 - \widetilde{\xi_2} (y_1) \xi_1 = 0$, which contradicts the fact that
  $\xi_1, \xi_2$ are linearly independent. Therefore we pick $y_2 \in \Omega
  \backslash A_2$. Iterating this process $q - 1$ times, we finally get
  an analytic subset
  \begin{equation}
    A_q \assign \{y \in \Omega ; \sum_{i = 1}^q (- 1)^{q - i} \alpha_i
    \tilde{\xi}_i (y) = 0\}
  \end{equation}
  with $\alpha_i = \det (\tilde{\xi}_s (y_t))_{s = 1, \ldots, \hat{i}, \ldots,
  q, t = 1, \ldots, q - 1}$ and $\alpha_q \neq 0$. Therefore, $A_q$ is proper since $\xi_1,...,\xi_q$ are linearly independent.
  One may pick $y_q \in \Omega \backslash A_{q }$ and $[\tilde{\xi}_i
  (y_k)]_{i, k = 1, \ldots, q}$ is non-degenerate.

  In summary, $t_{i j} = 0$ and $h^0 (X, \mathcal{E} \otimes f^{\ast}
  \mathcal{O}_Y (D_Y)) \geq q r$.
\end{proof}

Now we are in position to give the proof of Theorem \ref{mainthm}.

\begin{proof}
  (Proof of Theorem \ref{mainthm}) Let us first verify that for any semi-positive line bundle $H \rightarrow X$ and any $k \in
  \mathbb{N}$, some divisible $k_0 \in \mathbb{N}$ (such that $k_0 L$ is a
  line bundle) and some ample enough line bundle $G
  \rightarrow Y$,
  \begin{equation}
    \mathcal{F} \assign \mathcal{O}_Y (G) \otimes f_{\ast} \mathcal{O }_X ((k
    k_0 K_{X / Y} + k k_0 L) \otimes H \otimes \mathcal{I}_{{k k_0} } (h_L))
    \label{;;}
  \end{equation}
  is sufficiently generically globally generated. $H$ is to be determined at the end of the proof.

  Actually, It suffices to show that the restriction morphisms
$$
    \Gamma (X, (k k_0 K_{X / Y} + k k_0 L + f^{\ast} G + H) \otimes
    \mathcal{I}_{k k_0} (h_L)) \rightarrow $$
    \begin{equation}
    \Gamma \left( X_{y }, k k_0 \left(
    K_{X_{y }} + L|_{X_y} \right) \otimes \mathcal{O} (H|_{X_y}) \otimes
    \mathcal{\mathcal{I} }_{k k_0} (h_L |_{X_y}) \right) \label{32}
  \end{equation}
  are surjective for sufficiently general $y$. The reason why (\ref{32})
  implies (\ref{;;}) is that
  \begin{equation}
    \mathcal{F }_y = \Gamma \left( X_{y }, k k_0 \left( K_{X_{y }} + L|_{X_y}
    \right) \otimes \mathcal{O} (H|_{X_y}) \otimes \mathcal{\mathcal{I} }_{k
    k_0} (h_L) |_{X_y} \right) \otimes \mathcal{O}_{Y, y}
  \end{equation}
  for $y$ belonging to a Zariski dense open subset of $Y$ thanks to the Grauert's direct image and semi-continuity theorem (cf. Theorem 10.6 and Theorem 10.7 in \cite{bertin}), and
  \begin{equation}
    \mathcal{\mathcal{I} }_{k k_0} (h_L) |_{X_y} = \mathcal{\mathcal{I} }_{k
    k_0} (h_L |_{X_y})
  \end{equation}
  holds for sufficiently general $y \in Y$.

  The function $\psi$ can be constructed as in {\cite{zhou-zhu1}} as follows.
  Take a smooth concave function $\theta : \mathbb{R}
  \rightarrow (- \infty, 0]$ such that $\theta (t) = t$ when $t \in (-
  \infty, - 1]$, $\theta (t) = 0$ when $t \in [1,{\infty})$ and $0
  \leq \theta' \leq 1$, $\theta'' \geq - C_1$ for all $t \in \mathbb{R}$. Let
  $z = (z_1, \ldots, z_m)$ be the coordinate functions on a coordinate chart
  $U$ centered at $y$. Let $\psi$ be a global quasi-plurisubharmonic function
  on $Y$ defined by
  \begin{equation}
    \psi = \theta (\log |z - y|^{2 m}) . \label{podq}
  \end{equation}
  In fact, $\psi \leq 0$ and $i \partial \overline{\partial} \psi$ is only
  supported on $U' \assign U \bigcap \{z ; - e \leq |z - y| \leq e\}$ and
  \begin{equation}
    i \partial \overline{\partial} \psi |_{U'} \geq \theta''(\log |z - y|^{2 m}) \cdot (i \partial
    \overline{\partial} \log |z - y|^{2 m}) |_{U'} \geq - C_{2 } \omega |_{U'}
  \end{equation}
  for some K{\"a}hler metric $\omega$ and $C_2>0$. By compactness, for any $y\in Y$, we find a uniform constant $C_{3}$ independent of $y$, so that there exists a quasi-plurisubharmonic function $\psi$ such that $\psi\sim 2m\log|z-y|$ near $y$ and $i \partial
    \overline{\partial}\psi\geq -C_{3}\omega$.

  Now applying Lemma \ref{lemma11} (replace $L$ in Lemma \ref{lemma11} by $L
  + (k k_0)^{- 1} H$ equipped with the metric $h^{1 / k k_0}_H \cdot h_L$, where $h_H$ is semi-positive) we
  know that  (\ref{32})  holds true for some $G$ ample enough, which finishes our
  claim of  (\ref{;;}).

  Setting $h^0$ as the complex dimension of the cohomology group $H^0$ and
  using projection formula, one may observe that
$$
    h^0 (X, k k_0 (K_X + L) \otimes H \otimes f^{\ast} G \otimes
    \mathcal{I}_{k k_0} (h_L)) \geq \tmop{rk} \mathcal{F} \cdot h^0 (Y, k k_0
    K_Y)
$$
  \begin{equation}
    = h^0 (F, k k_0 (K_F + L_F) \otimes H|_F \otimes \mathcal{I}_{k k_0} (h_L
    |_F)) \cdot h^0 (Y, k k_0 K_Y) \label{38}
  \end{equation}
  holds for sufficiently general fiber $F$ thanks to Lemma \ref{addti}.

  In order to prove (\ref{112}), let us take a positive integer $k_1$ and a positive
  integer $C_4 $ such that
  \begin{equation}
    h^0 (Y, k k_1 K_Y) \geq C_4 k^{\kappa (Y, K_Y)} \label{311}
  \end{equation}
  holds for every integers $k \gg 0$. Then we obtain
  \[ h^0 (X, k k_2 (K_X + L) \otimes H \otimes f^{\ast} G \otimes
     \mathcal{I}_{k k_2} (h_L)) \geq \]
  \begin{equation}
    C_{5}  k^{\kappa (Y, K_Y)} \cdot h^0 (F, k k_2 (K_F + L_F) \otimes H|_F
    \otimes \mathcal{I}_{k k_2} (h_L |_F)) \label{312}
  \end{equation}
 holds for $k_2 = k_1 k_0$ and every positive enough integer $k$ according to
  (\ref{38}) and (\ref{311}), where $C_5 = C_4 k_0^{\kappa(Y,K_Y)}$. If we choose $H$ to be sufficiently ample then
  \[ h^0 (F, k_{\nu} k_2 (K_F + L_F) \otimes H|_F \otimes \mathcal{I}_{k_{\nu}
     k_2} (h_L |_F)) \geq \]
  \begin{equation}
    C ' k_{\nu}^{\kappa_{\sigma, k_2} (F, K_F + L_F, h_L |_F)} = C '
    k_{\nu}^{\kappa_{\sigma} (F, K_F + L_F, h_L |_F)} \label{combi}
  \end{equation}
  holds true with $C' >0$, by Lemma \ref{simple} and passing to some subsequence $\{k_{\nu}
  \}_{\nu=1}^{\infty}$. Finally, \ inequality (\ref{112}) follows immediately by the fact that
  for any $\nu\gg 0$
  \begin{equation}
    h^0 (X, k_{\nu} k_2 (K_X + L) \otimes H \otimes f^{\ast} G \otimes
    \mathcal{I}_{k_{\nu} k_2} (h_L)) \geq C_{5} C' k_{\nu}^{\kappa_{\sigma} (F,
    K_F + L_F, h_L |_F) + \kappa (Y, K_Y)},
  \end{equation}
  which is due to (\ref{312}) and (\ref{combi}).

 It remains to prove (\ref{112k}). For this purpose, by setting $D_Y=kk_0K_Y + (N-1)G$ and $H=\mathcal O_X$ in Lemma \ref{addti}, it follows that
$$
    h^0 (X, k k_0 (K_X + L)  \otimes f^{\ast} (NG) \otimes
    \mathcal{I}_{k k_0} (h_L)) \geq \tmop{rk} \mathcal{F} \cdot h^0 (Y, k k_0
    K_Y+(N-1)G)
$$
  \begin{equation}
    = h^0 (F, k k_0 (K_F + L_F) \otimes \mathcal{I}_{k k_0} (h_L
    |_F)) \cdot h^0 (Y, k k_0 K_Y+(N-1)G) \label{38k}
  \end{equation}
holds for sufficiently general fiber $F$, where $N$ is any positive integer (actually, we see from its proof that (\ref{38}) still holds if we replace $G$ by $NG$ for any positive integers $N$). Let us take another positive integer $\tilde{k}_{1}$ so that
  \begin{equation}\label{ipad}
  h^0(F,k\tilde{k}_{1}(K_F + L_F)\otimes\mathcal{I}_{k \tilde{k}_{1}}(h_L |_{F}))\geq  \tilde{C}_4 k^{\kappa (F, K_F + L_F, h_L |_F)}
  \end{equation}
  holds for every integers $k\gg0$ from the definition of $\kappa(F,K_F+L_F,h_{L}|_F)$. Therefore, thanks to (\ref{38k}) and (\ref{ipad}), we obtain that
  $$
    h^0(X,k\tilde{k}_{2}(K_X + L)\otimes f^{\ast}(NG)\otimes \mathcal{I}_{k\tilde{k}_2}(h_L))\geq
  $$
  \begin{equation}\label{ipad1}
 \tilde{C}_5  k^{\kappa (F, K_F + L_F, h_L |_F)} \cdot h^0 (Y, k \tilde{k}_2
    K_Y+(N-1)G)
  \end{equation}
  holds for $\tilde{k}_2 = \tilde{k}_1 k_0$ and every $k\gg 0$, where $ \tilde{C}_5= \tilde{C}_4 k_0^{\kappa (F, K_F + L_F, h_L |_F)}$. Now choose $N$ sufficiently large so that after extracting a subsequence $\{k_\nu\}_{\nu=1}^{\infty}$,
  \begin{equation}\label{ipad2}
  h^0 (Y, k_{\nu} \tilde{k}_2 K_Y+(N-1)G)\geq \tilde{C}'k_{\nu}^{\kappa_{\sigma}(Y,K_Y)}
  \end{equation}
  holds for every $\nu$ with $\tilde{C}'>0$. In conclusion, we finally get that
  \begin{equation}
    h^0 (X,k_{\nu} \tilde{k}_2 (K_X + L)\otimes f^{\ast} G \otimes
    \mathcal{I}_{k_{\nu} \tilde{k}_2} (h_L)) \geq \tilde{C}_{5} \tilde{C}' k_{\nu}^{\kappa (F,
    K_F + L_F, h_L |_F) + \kappa _{\sigma}(Y, K_Y)}
  \end{equation}
  for any $\nu\gg 0$ by (\ref{ipad1}) and (\ref{ipad2}), which confirms (\ref{112k}).
\end{proof}

\begin{remark}
From the proof of (\ref{112k}), we may observe that the left hand side in (\ref{112k}) can be replaced by a more optimal number $\kappa_{\sigma,f,\text{hor}}$ defined as in (\ref{jiangluo}).
\end{remark}

\begin{remark}[Analytic proof of Theorem \ref{thm1.1}]
  \label{alter}By the same token as above, we can also give an analytic method
  towards Theorem \ref{thm1.1}. For this purpose, we
  first claim that for any $k \in \mathbb{N}$
  \begin{equation}
    \mathcal{O }_Y (G) \otimes f_{\ast} \mathcal{O}_X (k K_{X / Y} + k (D_X -
    f^{\ast} D_Y) + H) \label{41}
  \end{equation}
  is generically globally generated for some ample line bundle $G$ over $Y$
  and any ample line bundle $H$ over $X$.

  Indeed, the projection formula yields that
  \[ \Gamma (Y, \mathcal{O }_Y (G) \otimes f_{\ast} \mathcal{O}_X (k K_{X / Y}
     + k (D_X - f^{\ast} D_Y) + H)) = \]
  \begin{equation}
    \Gamma (X, k (K_{X / Y} + D_X - f^{\ast} D_Y) + H + f^{\ast} G)
  \end{equation}
  holds. Since $D_X \supset f^{\ast} D_Y$, $D_X
  - f^{\ast} D_Y$ has simple normal crossings. Hence we can also construct $\psi$ as in
  (\ref{podq}) and choose $G$ and $H$ as in Lemma \ref{lemma111}. Then the
  general stalk of the locally free sheaf (\ref{41}) (whose restriction to the complement of some codim $\geq2$ analytic subset in $Y$ is locally free) becomes
  \begin{equation}
    \Gamma (X_y, (k (K_{X_y} + D_X |_{X_y}) + H|_{X_y})\otimes \mathcal{O}_{Y,y} \nobracket
  \end{equation}
  and our claim is an immediate consequence of Lemma \ref{lemma111} by setting
  $L = \mathcal{O}_X (D_X - f^{\ast} D_Y)$.

 The rest of the proof follows by Lemma
  \ref{addti} and the last half (i.e., computing and estimating $h^0$) of
  the proof of Theorem \ref{mainthm}. The only difference is that in order to prove (\ref{spck}), though the use of ample $H$ in (\ref{41}) is unavoidable, we can still apply inequality
  \begin{equation}
  h^0(F,k\tilde{k}_{1}(K_F + D_F)+H_F)\geq h^0(F,k\tilde{k}_{1}(K_F + D_F)) \geq  \tilde{C} k^{\kappa (F, K_F + D_F)}
  \end{equation}
 to substitute (\ref{ipad}).
\end{remark}

\begin{remark}[about the proof of Theorem  \ref{ne}]
Among the first half (i.e. the reduction to the surjective statement and the construction of $\psi$) of the proof of Theorem \ref{mainthm} and Remark \ref{alter} as above, we have seen that Theorem \ref{ne} can be obtained via Lemma \ref{lemma11} and Lemma \ref{lemma111} immediately.
\end{remark}

\section{Iitaka $D$-dimensions and their generalizations}\label{sect5}

The goal of this part is to recall three equivalent definitions of Iitaka
$D$-dimensions and generalize these results to generalized Kodaira
dimensions. The main references for this section are {\cite{mori1}}
and {\cite{ueno00}}.

\subsection{Iitaka $D$-dimension}\label{subsec5.1}

Let us first recall some facts in commutative algebra.

Let $R = \bigoplus_{k \in \mathbb{N}} R_k$ be a graded $\mathbb{C}$-algebra
and an integral domain (or briefly, $\mathbb{C}$-domain). Suppose that $R_k =
0$ for all $k < 0$ and $R_0 = \mathbb{C}$. Let $Q (R)$ be the quotient field
of $R$. Let $R^{(m)} \assign \bigoplus_{k \in \mathbb{N}} R_{k m}$ be the
graded $\mathbb{C}$-subdomain for some $m \in \mathbb{N}$. Let $R^{\#}$ denote
the multiplicative subset of all nonzero homogeneous elements.

The the quotient ring of $R$, denoted by $R^{\#- 1} R$, is also a graded
$\mathbb{C}$-domain. Its degree $0$ part $(R^{\#- 1} R)_0$ is also a field
which will be denoted by $Q ((R))$. It will be easy to check $Q ((R^{(m)})) = Q
((R))$ for any $m \in \mathbb{N}$.

Let $N (R) \assign \{k > 0 ; R_k \neq 0\}$ and let $M_{\geq n} \assign \{k \in
M ; k \geq n\}$ for any subset $M \subset \mathbb{N}$ and $n \in \mathbb{N}$.
Since $N (R)$ is a semigroup, we know that $N (R)_{\geq n} = (d
\mathbb{N})_{\geq n}$ for $n \gg 0$ and some $d = \gcd N (R) \in \mathbb{N}$.

Let $S = \bigoplus_{k \in \mathbb{N}} S_k$ be a graded $\mathbb{C}$-subalgebra
of $R$. For any $k \geq 0$, let $S_0 [S_k]$ denote the graded
$\mathbb{C}$-subdomain generated by $S_0$ and $S_k$.

The following propositions are a list of basic properties in commutative
algebra.

\begin{proposition}
  \label{begin}$R^{(m)} \subset R$ is an integral extension of rings for any
  integers $m \in \mathbb{N}$. In particular, the Krull dimensions of all
  $R^{(m)}$ are all the same.
\end{proposition}

\begin{proposition}
  If $R$ is finitely generated over $R_0 = \mathbb{C}$, then the Krull
  dimension of $R$ equals the transcendence degree of $Q (R)$ over
  $\mathbb{C}$.
\end{proposition}

\begin{proposition}
  If $R$ is finitely generated over $R_0$, then for all sufficiently divisible
  $d \in \mathbb{N}$, the graded subalgebra $R^{(d)}$ is finitely generated in
  degree $1$ over $R_0$ {\tmem{(i.e., $R^{(d)} = R_0 [R^{(d)}_1$])}}.
\end{proposition}

\begin{proposition}
  \label{en}If $R$ is finitely generated in degree $1$ over $R_0$, then the
  Hilbert function $\tmop{HF} : k \in \mathbb{N} \mapsto \dim_{\mathbb{C}} R_k
  \in \mathbb{N}$ has growth order $\delta - 1$, where $\delta$ is the Krull
  dimension of $R$. Precisely, this means $\tmop{HF} (k) \sim C \cdot
  k^{\delta - 1}$ as $k \rightarrow \infty$ for some $C > 0$.
\end{proposition}

Let $X$ be a smooth projective variety, $Q (X)$ be the rational function field of $X$ and $D$ be a Cartier divisor on
$X$. Let $L = \mathcal{O}_X (D)$ be the corresponding invertible sheaf of $D$.

From now on, we always set $R (X, L) : = R = \bigoplus_{k \in \mathbb{N}}
R_k$ and $$R_k \assign H^0 (X, \mathcal{O}_X (k D)) = H^0 (X, k L)$$ be the
associated $\mathbb{C}$-graded algebra. We also set $N (X, L) \assign N (R)$,
$Q ((X, L)) \assign Q ((R))$. Note $Q ((R))$ is algebraically
  closed in $Q (X)$ (cf. Proposition 1.4 in {\cite{mori1}}).

For any $k \in N (X, L)$, $R_k$ induces a Kodaira meromorphic mapping (cf. examples below Definition \ref{pv})
\begin{equation}
  \Phi_{|k L|} : X \dashrightarrow \mathbb{P} (H^0 (X, k L)^{}),
\end{equation}
which is a morphism outside the base locus. When $N (X, L) = \varnothing$, set
$\kappa (X, L)$ to be $- \infty$. When $N (X, L) \neq \varnothing$, the Iitaka
$D$-dimension $\kappa (X, L)$ or $\kappa (X, D)$ is defined as either one of
the following numbers.

\begin{proposition}\label{comp-1}
  \label{prk;}The following numbers are equal when $N (X, L) \neq
  \varnothing$:
  \begin{enumeratenumeric}
    \item $\kappa^{(1)} (X, L) \assign${\tmem{}}$\text{} \tmop{tr} .
    \deg_{\mathbb{C}} R (X, L) - 1 = \tmop{tr} . \deg_{\mathbb{C}} Q ((X,
    L))$;

    \item $\kappa^{(2)} (X, L) \assign \max_{k \in N (X, L)} \dim \Phi_{|k L|}
    (X)$;

    \item $\kappa^{(3)} (X, L) \assign \max \{\nu \in \mathbb{N} ; \limsup_{k
    \rightarrow \infty} \frac{h^0 (X, k L)}{k^{\nu}} > 0\}$.
  \end{enumeratenumeric}
\end{proposition}

Let us briefly explain the idea to the proof of Proposition \ref{comp-1}.
It follows easily from the fact that
    \begin{equation}
      Q ((X, L)) = Q (\tmop{Im} \Phi_{|k L|})
    \end{equation}
    for all $k \in N (X, K_X + L)$ sufficiently large (cf. Proposition 1.4 in {\cite{mori1}}) that $\kappa^{(1)}(X,L)=\kappa^{(2)}(X,L)$.

    To see the reason why $\kappa^{(2)}(X,L)=\kappa^{(3)}(X,L)$, there are, to the best of the authors' knowledge, at least two ways to show this result. On one hand, the first strategy relies on the existence of Iitaka fibration (cf. Theorem 1.12 and Corollary 1.13 in \cite{mori1}), which aims at reducing the problems on the line bundle $L$ to a big line bundle on a smooth model of $\tmop{Im} \Phi_{|kL|}$. On the other hand, by fixing a faithful $\mathbb{Z}^{n}$-valued valuation $v: Q(X)\backslash \{0\}\rightarrow \mathbb{Z}^n$ on $X$ a priori, the second strategy (cf. Theorem 3.3 + Corollary 1.16 in \cite{kk12}) considers a general graded algebra $R\subset Q(X)$ (in this case we set $R=R(X,L)$) of almost integral type and applies some basic facts from convex geometry (cf. Theorem 1.14 in \cite{kk12}) to obtaining that both $\kappa^{(2)}(X,L)$ and $\kappa^{(3)}(X,L)$ coincide with the dimension of its Newton convex body of $\Delta_{v}(X,L)$ related to $R=R(X,L)$. We will see in the next subsection that how these two methods can be generalized, when considering the graded subalgebra containing multiplier ideal sheaf and assuming $X$ is merely compact (i.e. possibly non-algebraic).

In the special case when $R(X,L)$ is a finitely generated algebra, $\kappa^{(2)}(X,L)=\kappa^{(3)}(X,L)$ can be proved by directly applying Proposition \ref{begin}-\ref{en}.  In fact, if $R(X,L)$ is finitely generated, then by choosing $d$ sufficiently divisible, 
we obtain
    $\kappa^{(3)}(X,L) = \kappa^{(3)}(X, dL) = \dim R^{(d)} - 1 = \tmop{tr} .
    \deg_{\mathbb{C}} Q ((R^{(d)})) = \kappa^{(2)}(X, dL) = \kappa^{(2)}(X,L)$.

\subsection{Generalized Kodaira dimension}

In this part, we will still adopt the notation as in subsection
\ref{subsec5.1}. Given a vector space $V$, the notation $\mathbb{P} V$ stands
for all $1$-dimensional quotients of $V$ or all hyperplanes of $V^{\ast}$.

Let $X$ be a compact (connected) complex manifold. Let $L$ be
a pseudoeffective line bundle over $X$ equipped with a singular hermitian
metric $h_L$ whose curvature current is semi-positive. The following results in
this subsection can all be generalized when $L$ is a $\mathbb{Q}$-bundle.

Let $\mathcal{I}_k (h_L)$ be the $k$-multiplier ideal sheaf as in
(\ref{ddjo}), then
\begin{equation}\label{aftd}
S_k := H^0 (X, k (K_X + L) \otimes \mathcal{I}_k (h_L))
\end{equation}
contains
all holomorphic sections $s$ of $\mathcal{O}_X (k (K_X + L))$ on $X$ and
\begin{equation}
  \|s\|_k \assign (\int_X |s|^{2 / k} h_L)^{k / 2} < \infty . \label{1830}
\end{equation}
Note although at each $k$-level $\| \cdot \|_k$ does not satisfy the triangle
inequality (hence not a norm), it still can be checked that
\begin{equation}
  \|s_1 + s_2 \|_k \leq 2^{{k}} (\|s_1 \|_k +\|s_2 \|_k)
  \label{daffsa}
\end{equation}
holds for any $s_1, s_2 \in S_k$.

From now on, let us take $R _k = H^0(X, k(K_X +L))$
and $R=\bigoplus_{k\in\mathbb{N}}R_k$. Set
\begin{equation}
  S = \bigoplus_{k \in \mathbb{N}} S_k \subset R, \label{notj}
\end{equation}
where $S_k$ is defined as in (\ref{aftd}).
By the basic fact that $\mathcal{I}_k (h_L)$ is a coherent sheaf, it is
recognized that $S_k$ is a linear subspace of $R_k$. Since for $s_{m}\in S_m$ and $s_l\in S_l$
\begin{equation}
  \label{jdil} \|s_m \cdot s_l \|_{m + l} \leq \|s_m \|_m \cdot \|s_l \|_l
\end{equation}
holds by H{\"o}lder inequality, $S$ becomes a graded  $\mathbb{C}$-subalgebra of $R$.

Let $ Q ((X, K_X + L, h_L))$ and $ N
(X, K_X + L, h_L)$ denote the  degree zero part of the quotient ring of $S$ and respectively the corresponding semigroup of $S$ (see section 4.1). Let $\mathbb{P} H^0 (X, k (K_X + L)
\otimes \mathcal{I}_k (h_L))$ be the projective space of all hyperplanes of
$H^0 (X, k (K_X + L) \otimes \mathcal{I}_k (h_L)) .$

\begin{definition}\label{pv}
  {\tmem{(Definition 2.2 in {\cite{ueno00}})}} Let $X, Y$ be two complex
  spaces. A mapping from $X$ to the power set of $Y$, denoted by $\Phi : X
  \dashrightarrow Y$, is called a {\tmstrong{{\tmem{meromorphic mapping}}}},
  if the followings are satisfied:
  \begin{enumeratenumeric}
    \item The graph $G \assign \{(x, y) \in X \times Y ; y \in \varphi (x)\}$
    is an irreducible analytic subset in $X \times Y$;

    \item The projection map $p_X : G \rightarrow X$ is a proper modification.
  \end{enumeratenumeric}
\end{definition}

Take $k \in N (X, K_X + L, h_L)$, then basic examples of meromorphic mappings
\begin{equation}
  \Phi_{|S_k |} : X \dashrightarrow \mathbb{P} H^0 {(X, k (K_X + L) \otimes
  \mathcal{I}_k (h_L))=\mathbb{P} S_k }  \label{expl}
\end{equation}
will be constructed as follows (cf. Example 2.4.1 and 2.4.2 in
{\cite{ueno00}} when $\mathcal{I}_k (h_L)$ is trivial).

Let $\{s_0, \ldots, s_{N_k} \}$ be a basis of $S_k$ and set
\begin{equation}
  \Sigma \assign \{x \in X ; s_0 (x) = \cdots = s_{N_k} (x) = 0\}
\end{equation}
to be a nowhere dense analytic subset of $X$. Then there is a holomorphic
mapping
\begin{equation}
  \Phi: X \backslash \Sigma \rightarrow \mathbb{P}^{N_k}, x \mapsto \{s \in S_k ; s (x) = 0\}
\end{equation}
and $G$ be the closure of the graph of $\Phi$ in $X \times
\mathbb{P}^{N_k}$. One can conclude that $G$ is an (irreducible) analytic subset
and $p_X : G \rightarrow X$ is a proper modification. Hence $G$ induces a
meromorphic map denoted by (\ref{expl}). Let $\tmop{Im} \Phi_{|S_k |}$ be the image of the projection of $G$ onto
$\mathbb{P}^{N_k}$. By Remmert's proper mapping theorem, $\tmop{Im} \Phi_{|S_k |}$
is a projective algebraic variety in $\mathbb{P}^{N_k}$.

\begin{proposition}
  \label{pdju}$Q (\tmop{Im} \Phi_{|S_k |}) = Q ((X, K_X + L, h_L))$ for all $k
  \in N (X, K_X + L, h_L)$ sufficiently large enough.
\end{proposition}

\begin{proof}
  By the very basic fact that the (meromorphic) function field $Q (X)$ is
  finitely generated {\tmem{{\tmem{(cf. Theorem 3.1 in {\cite{ueno00}})}}}},
  any subfield of $Q (X)$ will be finitely generated. As a consequence, $Q (X,
  K_X + L, h_L)$ is finitely generated over  $\mathbb{C}$. Then there exists $M > 0$ so
  that $Q ((X, K_X + L, h_L)) = Q ((S_0 [S_k])) = Q (\tmop{Im} \Phi_{|S_k |})$
  for all $k \in N (X, K_X + L, h_L)_{\geq M}$ {\tmem{{\tmem{(cf. (1.2 i) in
  {\cite{mori1}})}}.}}
\end{proof}

\begin{proposition}
  \label{alcl}$Q ((X, K_X + L, h_L))$ is algebraically closed in $Q (X)$.
\end{proposition}

\begin{proof}
  Let us first show $S$ is integrally closed in $R$. Take any $r \in R_{\nu}
  (\nu \in \mathbb{N})$ satisfying the following equation
  \begin{equation}
    r^k + s_1 r^{k - 1} + \cdots + s_{k - 1} r + s_k = 0 \label{jjol}
  \end{equation}
  where $s_j \in S_{\nu j}$ for $j = 1, \ldots, k$, we claim that $r \in
  S_{\nu}$ or equivalently
  \begin{equation}
    \|r\|_{\nu} \leq 2^{{\nu k^2}} (k + 1) \max_{j = 1, \ldots k}
    \{\|s_j \|_{j \nu}^{1 / j} \} < \infty . \label{1980}
  \end{equation}
  \quad Indeed, if by contradiction {\tmem{{\tmem{(\ref{1980})}}}} does not
  hold, then
  \begin{equation}\label{compp}
  \|r\|_{\nu}^j > 2^{{\nu k^2}}  (k + 1) \|s_j \|_{ \nu j}
  \end{equation}
  holds for any $j \in [1, k]$. According to {\tmem{{\tmem{(\ref{daffsa})}}
  }}and equation{\tmem{ {\tmem{(\ref{jjol})}}}}, we get
  \[ \|r\|^k_{\nu} = \|r^k \|_{\nu k} \leq 2^{{\nu k}} (\|s_1 r^{k -
     1} \|_{\nu k} +\|r^k + s_1 r^{k - 1} \|_{\nu k}) \]
  \[ \leq 2^{{\nu k}} (\|s_1 r^{k - 1} \|_{\nu k} + 2^{{\nu
     k}} (\|r^k + s_1 r^{k - 1} + s_2 r^{k - 2} \|_{\nu k} +\|s_2 r^{k - 2}
     \|_{\nu k})) \leq \ldots . \]
  \begin{equation}
    \leq 2^{{\nu k^2}} (\|s_1 r^{k - 1} \|_{\nu k} + \cdots +\|s_k
    \|_{\nu k}) \label{jjik} .
  \end{equation}
  It follows from {\tmem{{\tmem{(\ref{jdil}),}}}}
  {\tmem{{\tmem{(\ref{1830})}}}} and (\ref{compp}) that
  \begin{equation}
    \|s_j r^{k - j} \|_{\nu k} \leq \|s_j \|_{\nu j} \cdot \|r^{k - j} \|_{\nu
    (k - j)} = \|s_j \|_{\nu j} \cdot \|r  \|_{\nu}^{k - j} \leq
    \frac{1}{2^{{\nu k^2}}  (k + 1)} \|r\|^k_{\nu} \label{kkod}
  \end{equation}
  for $j \in [1, k]$. Thanks to
  ({\tmem{{\tmem{{\tmem{{\tmem{\ref{jjik})}}}}}}}} and
  {\tmem{{\tmem{(\ref{kkod})}}}}, we eventually get the contradiction:
  \begin{equation}
    \|r\|^k_{\nu} \leq 2^{{\nu k^2}} \cdot k \cdot
    \frac{1}{2^{{\nu k^2}}  (k + 1)} \|r\|^k_{\nu} = \frac{k}{k + 1}
    \|r\|^k_{\nu}
  \end{equation}
  as $r\neq 0$. Hence {\tmem{{\tmem{(\ref{1980})}}}} has been proved.

  The rest of the proof will be almost the same as Proposition 1.4
  in{\tmem{{\tmem{ {\cite{mori1}}}}}}. Let $\xi$ be the $Q (X)$-scheme
  $\tmop{Spec} Q (X)$ and $\xi \rightarrow X$ be the inclusion (see Ex2.7 in GTM52 for the details of this map). Let $(K_X +
  L)_{\xi}$ be the pull back of $K_X + L$. Since $S$ is integrally closed in
  $R$, $R$ is integrally closed in $R (\xi, (K_X + L)_{\xi})$, we know that
  $S$ is integrally closed in $R (\xi, (K_X + L)_{\xi})$. Therefore, $Q ((S))$
  is algebraically closed in $S^{\#- 1} R (\xi, (K_X + L)_{\xi})$ by $(1.2)$
  in {\tmem{{\tmem{{\cite{mori1}}}}}}. Note $Q (X) = R (\xi, (K_X +
  L)_{\xi})_0 \subset S^{\#- 1} R (\xi, (K_X + L)_{\xi})$, we can now obtain
  that $Q ((S))$ is algebraically closed in $Q (X)$.
\end{proof}

Consider the following three numbers when $N (X, K_X + L, h_L) \neq
\varnothing$:
\begin{equation}\label{dko}
  \kappa^{(1)} (X, K_X + L, h_L) : = \tmop{tr} . \deg_{\mathbb{C}} Q ((X, K_X
  + L, h_L))
\end{equation}
\begin{equation}
  \kappa^{(2)} (X, K_X + L, h_L) : = \max_{k \in N (X, K_X + L, h_L)} \dim
  \tmop{Im} \Phi_{|S_k |} \label{jkp}
\end{equation}
\begin{equation}
  \kappa^{(3)} (X, K_X + L, h_L) : = \max \{m \in \mathbb{N} ; \limsup_{k
  \rightarrow \infty} \frac{h^0 (X, k (K_X + L) \otimes \mathcal{I}_k
  (h_L))}{k^m} > 0\} \label{pkj} .
\end{equation}
Note $ \kappa^{(3)} (X, K_X + L, h_L)$ has already been introduced in \cite{zhou-zhu1}.

The following proposition illustrates how these numbers vary under blow-ups.

\begin{proposition}
  \label{oo}Let $\mu : X' \rightarrow X$ be a blow-up of $X$ with smooth center $Z$ of $\tmop{codim} \geq 2$, then inequality
  \begin{equation}
    \kappa^{(j)} (X, K_X + L, h_L) = \kappa^{(j)} (X', K_{X'} + L', h_{L'})
    \label{21}
  \end{equation}
  holds for $j = 1, 2, 3$, where $(L', h_{L'}) = (\mu^{\ast} L, \mu^{\ast}
  h_L)$.
\end{proposition}
\begin{proof}
  By definition of $\kappa^{(j)} (X, K_X + L, h_L)$, since $H^0 (X', k (K_{X'}
  + L') \otimes \mathcal{I}_k (h_{L'})) = H^0 (X, \mu_{\ast} (k  K_{X'}
  \otimes k  L' \otimes \mathcal{I }_k (h_{L'} )))$, it will be sufficient for
  us to show
  \begin{equation}
    \mu_{\ast} (\mathcal{O}_{X'} (k  K_{X'} + k  L') \otimes \mathcal{I }_k
    (h_{L'} )) = \mathcal{O}_X (k  K_X + k  L) \otimes \mathcal{I }_k (h_L )
    \label{22}
  \end{equation}
  for any $k \in \mathbb{N}$. Let us write $h_L = e^{- 2\varphi_L}$ on some
  open coordinate $U \subset X$ which trivializes $L$, then $h_{L'} = e^{- 2
  \mu^{\ast} \varphi_L}$ on $\mu^{- 1} (U)$.

  By the definition of multiplier ideal sheaves, given any open subset $U\subset X$, the coherent analytic sheaf $\mathcal{O}_X (k K_X + k L)
  \otimes \mathcal{I }_k (h_L )(U)$ consists of all holomorphic $k$-canonical (i.e. $K_{U}^{\otimes k}$-valued) forms $f$ on $U$ such that
  \begin{equation}
    \int_V c_n (f \wedge \overline{f})^{1 / k} e^{- 2 \varphi_L} < \infty
    \label{190}
  \end{equation}
  holds for any $V \Subset U$. The change of variable formula yields that
  \begin{equation}
    \int_{\mu^{- 1} (V)} c_n (\mu^{\ast} f \wedge \overline{\mu^{\ast} f})^{1
    / k} e^{- 2 \mu^{\ast} \varphi_L} = \int_V c_n (f \wedge
    \overline{f})^{1 / k} e^{- 2 \varphi_L} < \infty \label{140} .
  \end{equation}
  Hence $\mu^{\ast} f$ is the pull-back holomorphic $k$-canonical form on $\mu^{- 1}
  (U)$ such that (\ref{140}) holds. Therefore, $\mu_{\ast} (k  K_{X'} \otimes
  k  L' \otimes \mathcal{I }_k (h_{L'} )) \supset k  K_X \otimes k  L \otimes
  \mathcal{I }_k (h_L)$.

  On the other side, if we a priori know that a holomorphic $k$-canonical form $\mu^{\ast} f$ on $\mu^{- 1} (U)$ satisfies (\ref{140}), then $f$
  must be a holomorphic $k$-canonical form on $U$ by the fact that $\mu$ is an isomorphism outside $\mu^{-1}(Z)$. Actually, this can be done according to the codim $\geq 2$
  extension theorem for holomorphic functions. Moreover, it must also satisfy
  (\ref{190}). As a result, $\mu_{\ast} (k  K_{X'} \otimes k  L' \otimes
  \mathcal{I }_k (h_{L'} )) \subset k  K_X \otimes k  L \otimes \mathcal{I }_k
  (h_L)$ holds. This finishes the proof of (\ref{22}).
   \end{proof}

Now let us turn to compare $\kappa^{(1)}, \kappa^{(2)}$ and $\kappa^{(3)}$. Our next goal is to show these three numbers all coincide with the dimension of an appropriate Newton-Okounkov body associated to $X$, $L$ and $h_L$, which also equals the so-called generalized Kodaira dimension. To this end, let us first recall some basic constructions of the Newton-Okounkov body of any given graded algebra $S\subset Q(X)$, where $Q(X)$ is the rational function field of a projective manifold $X$. A systematic study of the Newton-Okounkov body can be referred to \cite{lm09} and \cite{kk12}, which was based on the pioneering work of A. Okounkov (cf. \cite{ok96}). The following notations and basic facts are mainly taken from \cite{kk12}.

Let $P\subset \mathbb{Z}^n$ be a semigroup. Let $G$ be the subgroup of $\mathbb{Z}^n$ generated by $P$, $L$ be the subspace of $\mathbb{R}^n$ spanned by $P$. Let $C$ be the smallest closed convex cone with apex at the origin generated by $P$. The regularization $\tilde{P}$ of $P$ is defined to be the semigroup $G\bigcap C$ contained in $L$. Now assume that the cone $C$ is strongly convex (i.e. the linear subspace contained in $C$ is only zero) and set $\dim L=q+1$. Fix a rational half-space $M\subset L$ (i.e. $\partial M$ can be spanned by rational vectors) containing $P$. Take a linear map $\pi_M:L\rightarrow \mathbb{R}$ so that $\text{ker } \pi_M = \partial M$, $\pi_M (L\bigcap\mathbb{Z}^n)=\mathbb{Z}$ and $\pi_M (M\bigcap \mathbb{Z}^n)=\mathbb{Z}_{\geq 0}$.

Let $H_P (k):= \text{Card}(P\bigcap\pi_M^{-1}(k))$ and $H_{\tilde{P}}(k):= \text{Card}(\tilde{P}\bigcap \pi_M^{-1}(k))$ be the Hilbert function of $P$ and $\tilde{P}$ respectively, where $\text{Card}(\cdot)$ stands for the cardinality of a set. Then the convex body (i.e. a compact convex set, which can be confirmed due to the fact that $C$ is strongly convex) $\Delta(P):=\tilde{P}\bigcap \pi_M^{-1}(1)$ of dimension $q$ is called the \emph{Newton-Okounkov} body of the semigroup $P$. Let us fix the notations $m(P),\text{ind}(P)$ standing for the index of the subgroup $\pi_{M}(G)$ in $\mathbb{Z}$, and the index of the subgroup $G \bigcap \partial M$ in $\mathbb{Z}^{n-1}\times \{0\}$ respectively.

Let $X$ be a projective manifold and $Q(X)$ be the rational function field. Let $S=\bigoplus_{k\in\mathbb{N}} S_k\subset Q(X)$ be a graded algebra of \emph{almost integral type} (see the precise definition in section 2.3 in \cite{kk12}). The basic example for such $S$ includes all graded subalgebras of $R(X,L)$ associated to an arbitrary line bundle $L$ over $X$ (cf. Theorem 3.7 in \cite{kk12}). Fix a faithful $\mathbb{Z}^n$-valued valuation $v:Q(X)^{\ast}\rightarrow \mathbb{Z}^n$ (i.e. $v(Q(X)^{\ast})=\mathbb{Z}^n$) with respect to the total ordering of $\mathbb{Z}^n$ (e.g. for $p=(p_1,...,p_n),q=(q_1,...,q_n)\in \mathbb{Z}^n$, say $p>q$ iff for some $1\leq r <n$ we have $p_{i}=q_i$ for $i=1,...,r$ and $p_{r+1}>q_{r+1}$). Then $v$ naturally induces a valuation $v_t : Q(X)[t]^{\ast}\rightarrow \mathbb{Z}^{n+1}$ extending $v$ (see the details of the construction in section 2.4 in \cite{kk12}). The valuation $v_t$ maps non-zero elements of $S[t]:=\bigoplus_{k\in\mathbb{N}} S_k t^k$ to a semigroup, denoted by $P(S)$, of integral points contained in $\mathbb{Z}^n \times \mathbb{Z}_{\geq 0}$. Fix the rational half-space $M=\mathbb{R}^n\times\mathbb{R}_{\geq 0}$ and take $\pi_M:\mathbb{R}^{n+1}\rightarrow \mathbb{R}$ to be the projection map to the last coordinate. It can be shown that the cone associated to $P(S)$ is strongly convex since $S$ is of almost integral type (cf. Theorem 2.30 in \cite{kk12}). Then one can define the convex body $\Delta(P(S))=\tilde{P}(S) \bigcap (\mathbb{Z}^n\times\{1\})$ to be the \emph{Newton-Okounkov} body, denoted by $\Delta(S)$, of the algebra $S$. Let us fix the notations $m(S),\text{ind}(S)$ standing for the indices $m(P(S))$, $\text{ind}(P(S))$ for the semigroup $P(S)$ respectively.

We will adopt the following two important properties of Newton-Okounkov bodies in our context.
\begin{lemma}[Theorem 2.31 in \cite{kk12}]\label{resp}
Let $S=\bigoplus_{k\in\mathbb{N}}S_k$ be an algebra of almost integral type with the Newton-Okounkov body $\Delta(S)$. Put $m=m(S)$, $q=\dim \Delta(S)$ and the Hilbert function $H_S (k):=\dim S_k$. Then the $q$-th growth coefficient of the function $H_S (m\cdot)$
$$    a_q:=\lim_{k\rightarrow\infty}\frac{H_S (mk)}{k^q}    $$
exists and $a_q = \emph{Vol}_q (\Delta(S))/\emph{ind} (S)>0$.
\end{lemma}

\begin{lemma}[a part of Theorem 3.3 in \cite{kk12}]\label{lema41kk12}
Let $S$ be an algebra of almost integral type in $Q(X)$ and $S_k$ be the $k$-th subspace of the algebra $S$. Let $Y_k:=\emph{Im} \Phi_k$, where $\Phi_k$ is the Kodaira map defined in \emph{(\ref{expl})} associated to the linear subspace $S_k$. If $p$ is sufficiently large and divisible by $m(S)$, then $\emph{dim} Y_p$ is independent of $p$ and equals the dimension of the Newton-Okounkov body $\Delta(S)$.
\end{lemma}
Let $X$ now be a (possibly non-algebraic) compact complex manifold. The algebraic dimension $a(X)$ of $X$ is defined to be the transcendental degree of $Q(X)$, namely the field of meromorphic functions, over $\mathbb{C}$.

\begin{definition}[Algebraic reduction, cf. p. 25 in \cite{ueno00}]\label{jlp}
A surjective morphism $a:X'\rightarrow A$ is called an \text{algebraic reduction} if it satisfies the following conditions:
\begin{itemize}
\item[1)] $X'$ is also a complex manifold and bimeromorphically equivalent to $X$.
\item[2)] $A$ is a projective manifold and $\dim A=a(X)$.
\item[3)] $a$ induces an isomorphism between $Q(X)$ and $Q(A)$.
\end{itemize}
\end{definition}

An algebraic reduction always exists (cf. p. 24-25 in \cite{ueno00}), and is unique up to a bimeromorphic equivalence. In addition, by the very basic fact Corollary 1.10 in \cite{ueno00}, the fibers of $a:X'\rightarrow A$ are connected (i.e. $a:X'\rightarrow A$ is an analytic fiber space) from the above condition 3).

Let $F$ be a holomorphic line bundle over $X$.
The following result will be a useful tool for the study of the asymptotic behavior of a graded linear system associated to $X$ and $F$. For completeness, we will also give a sketch of proof of the following lemma, whose original idea comes from \cite{wxj22}.
\begin{lemma}[Theorem 1 in \cite{wxj22}]\label{wxj22}
Let $a_0 \in\mathbb{N}$ such that $a_{0}F$ is effective. Then there exists a smooth projective variety $A$ (independent of $a_0$) and an algebraic reduction of $X$ such that there exists a $\mathbb{Q}$-effective divisor $D$ over $A$ and
\begin{equation}
H^0(A,kD)=H^0(X,k a_0 F)
\end{equation}
holds  for $k>0$ sufficiently divisible.
\end{lemma}

\begin{proof}
With the existence of the algebraic reduction of $X$ (cf. Definition \ref{jlp}) and the neat model of any holomorphic fibrations between compact complex manifolds (cf. Lemma 1.3 in \cite{campana} or Lemma 2 in \cite{wxj22}), we may assume there exists a proper modification $\mu: X'\rightarrow X$ and an algebraic reduction $a: X' \rightarrow A$, so that every $a$-exceptional divisor is also $\mu$-exceptional. We decompose the divisor $\mu^{\ast}(a_0 F)$ by (the dimension of the image of each irreducible component of $\mu^{\ast}F$ under the algebraic reduction $a$)
\begin{equation}\label{sierwu}
\mu^{\ast}(a_0 F)= N'+\sum_{j=1}^{J} b_j D'_j +R',
\end{equation}
where $N'$ is an effective divisor so that $a(N'_0)=A$ for each irreducible component $N'_0$ of $N'$ (thus $N'$ is non-polar, in the sense of Definition 1.2 in \cite{campana}), each $D'_j$ is a prime divisor whose image under $a$ is of codimension one (hence also a prime divisor in $A$, according to section 9.1.3 in \cite{grauert-remmert}) with coefficients $b_j\in\mathbb{Z}_{>0}$ and $R'$ stands for an $a$-exceptional divisor.

Let us denote $D_j:=a(D'_j)$ for $j=1,...,J$ and assume all $D_{l}$ ($l=1,...,L$) are not equal to each other with $L\leq J$, while $D_{j_1}$ ($j_1 >L$) equals one of the $D_{j_2}$ ($j_2 =1,...,L$) (since it might happen that $D_p = D_q$ for $p\neq q \in \{1,...,J\}$). Then we consider the following decomposition
\begin{equation}\label{sierliu}
a^{\ast}D_l=\sum_{j=1}^{J} c_{j,l} D'_j + \sum_{m=1}^{M_l} d_{m,l} E'_l +R'_l, l=1,...,L
\end{equation}
with prime divisors $E'_{l}$ for $l=1,...,L$ which do not contain any $D'_j$ for $j=1,...,J$, $a$-exceptional divisors $R'_l$ for $l=1,...,L$ and all integers $c_{j,l}\geq 0, d_{m,l} > 0$ and $c_{j,l}>0$ for at least one $j\in \{1,...,J\}$. For $j\in \{1,...,J\}$, we set $f_{j,l}=b_j$ if $c_{j,l}>0$ and $f_{j,l}=0$ if $c_{j,l}=0$. Then
\begin{equation}\label{sierqi}
G'_l:=\sum_{j=1}^{J} f_{j,l} D'_j >0
\end{equation}
will be the maximal divisor whose support is contained in all $D'_j$ ($j=1,...,J$) and $a^{\ast} D_l$, so that  $\mu^{\ast}(a_0 F)-G'_l$ is still effective.

Consider the $\mathbb{Q}$-effective divisor
$$ D:=\sum_{l=1}^{L} g_l D_l $$
with $ g_{l}=\min_{j=1,...,J}\frac{b_{j}}{c_{j,l}}\in\mathbb{Q}$ if $M_l =0$; $g_l =0$ if $M_l\neq 0$ ($l=1,...,L$). If $g_{l}=0$, then it can be shown that $G'_l$ will be partially supported on the fibers of $a$ (see its definition in Definition 1.21 in \cite{campana}). In this case, we set $T'_l:=0$. If $g_{l}>0$, then all  $f_{j,l}$ are not zero. Set $T'_l:=g_l R'_l$ then $G'_l -g_l \mu^{\ast}(D_l)+T'_{l}$ will also be a $\mathbb{Q}$-effective divisor partially supported on the fibers of $a$, thanks to (\ref{sierliu}) and (\ref{sierqi}). Denote by $T':=\sum_{l=1}^{L}T'_l$ the $a$-exceptional (hence also $\mu$-exceptional) $\mathbb{Q}$-effective divisor. Summing up with the index $l=1,...,L$, thus we obtain $P':=\mu^{\ast}(a_0 F)-N'-a^{\ast}D+T'$ will be a $\mathbb{Q}$-effective divisor partially supported on the fibers of $a$ according to (\ref{sierwu}).

Choosing $k\in\mathbb{N}$ sufficiently divisible, we will see that
$$
H^0(X,ka_0 F)=H^0(X',k\mu^{\ast}(a_0 F)+k T')=H^0(X',kP'+kN'+ka^{\ast}D)=H^0(A,kD)
$$
holds, by the basic property of exceptional divisors (cf. Lemma 3 in \cite{wxj22}),  non-polar divisors (cf. Lemma 4 in \cite{wxj22}) and partially supported divisors (cf. Lemma 1 in \cite{wxj22} or Lemma 1.22 in \cite{campana}).
\end{proof}

We are now in position to show that the aforementioned three numbers defined from (\ref{dko}) to (\ref{pkj}) all coincide, which will all be denoted by $\kappa(X,K_X +L, h_L)$ from now on.

\begin{proposition}
  \label{eqji}$\kappa^{(1)} (X, K_X + L, h_L) = \kappa^{(2)} (X, K_X + L,
  h_L)= \kappa^{(3)} (X, K_X + L, h_L)$.
\end{proposition}
  \begin{proof}
    The first equality follows easily from the fact that
    \begin{equation}
      Q ((X, K_X + L, h_L)) = Q (\tmop{Im} \Phi_{|S_k |})
    \end{equation}
    for all $k \in N (X, K_X + L, h_L)$ sufficiently large.

        Let us then show that
        \begin{equation}\label{harv}
         \kappa^{(2)} = \kappa_a^{(2)}, \kappa^{(3)} =
    \kappa_a^{(3)}
    \end{equation}
     for any integers $a \in \mathbb{N}$, where
    \begin{equation}
      \kappa_a^{(2)} \assign \max_{k \in \mathbb{N}} \dim \tmop{Im}
      \Phi_{|S_{a k} |}
    \end{equation}
    and
    \begin{equation}
      \kappa_a^{(3)} \assign \max \{m \in \mathbb{N} ; \limsup_{k \rightarrow
      \infty} \frac{h^0 (X, a k (K_X + L) \otimes \mathcal{I}_{a k}
      (h_L))}{k^m} > 0\} .
    \end{equation}
    For $\kappa^{(2)}$, one just need to use the fact $Q ((S^{(a)})) = Q
    ((S))$ and the first equality $\kappa^{(2)} (X, K_X + L, h_L) =
    \kappa^{(1)} (X, K_X + L, h_L) = \tmop{tr} . \deg_{\mathbb{C}} Q ((S)) .$ For $\kappa^{(3)}$, by definition, it is clear that
    $\kappa^{(3)} \geq \kappa^{(3)}_a$. On the
    other side, if there exists $\{k_{\nu} \}_{\nu = 1}^{\infty}$ such that
    \begin{equation}
      N_{\nu} \assign h^0 (X, k_{\nu} (K_X + L) \otimes \mathcal{I}_{k_{\nu}}
      (h_L)) \geq C k^{\kappa^{(3)}}_{\nu}
    \end{equation}
    for some $C > 0$, then $h^0 (X, a k_{\nu} (K_X + L) \otimes \mathcal{I}_{a
    k_{\nu}} (h_L)) \geq N_{\nu} \geq C k^{\kappa^{(3)}}_{\nu}$ and
    $\kappa^{(3)} \leq \kappa^{(3)}_a$. Hence $\kappa^{(3)} = \kappa^{(3)}_a$.

    Applying Lemma \ref{wxj22} with $F=K_X +L$, there exist a smooth projective variety $A$ and a $\mathbb{Q}$-effective divisor $D$ on $A$ so that $H^0(A,kD)=H^0(X,k a_0 (K_X+L))=H^0(X',k a_0 (K_{X'}+L'))$ holds for some integer $a_0>0$, any $k\in\mathbb {N}$ sufficiently divisible. Furthermore, the following diagram
      \begin{equation}
    \text{} \begin{array}{ccc}
      X' &
      \stackrel{\mu}\longrightarrow & X\\
      \downarrow  &  \\
      A    \end{array}
  \end{equation}
     holds, where $\mu:X' \rightarrow X$ is a proper modification and $X'\rightarrow A$ is a surjective holomorphic map with connected fibers. We may assume $H^0(A,kaD)=H^0(X,k a a_0(K_X+L))$ holds for some other integer $a$ and any integers $k$. 

     Therefore, $S^{(a a_0)}\subset W:=\bigoplus_{k\in\mathbb{N}} H^0(A,kaD)$ will be an algebra of almost integral type according to Theorem 3.7 in \cite{kk12}; thus we can consider $\Delta(S^{(a a_0)})$, which is defined to be the Newton-Okounkov body of $S^{(a a_0)}$. We highlight that $\Delta(S^{(a a_0)})$ might depend on the choice of the algebraic reduction and $A$ as pointed out in Remark 1 in \cite{wxj22} (however, its dimension does not rely on the choice of $A$).

     Identifying $S'_{k a a_0}:=H^0(X',ka a_0 (K_{X'}+L')\otimes \mathcal{I}_{k a a_0}(h_{L'}))$ as a linear subspace $V_k$ of $H^0(A,k aD)$ via an isomorphism $\theta_k$ for any $k\in\mathbb{N}$ (note $S'_{k a a_0}=S^{(a a_0)}_{k}$ by Proposition \ref{oo}), then $V_k V_l\subset V_{k+l}$ and we have the following diagram 
           \begin{equation}
    \text{} \begin{array}{ccc}
      X' & \stackrel{\Phi_{|S' _{ka a_0}|}}\dashrightarrow & \text{Im} \Phi _{|S' _{ka a_0}|}\subset \mathbb{P}^N\\
      \downarrow  &  & \theta_k^{\ast} \downarrow \\
      A  &
      \stackrel{\Phi_{|V _k|}}\dashrightarrow & \text{Im} \Phi _{|V _k|}\subset \mathbb{P}^N  \end{array}
  \end{equation}
  with $\theta_k^{\ast}$ an isomorphism between $\text{Im} \Phi _{|S' _{ka a_0}|}$ and $\text{Im} \Phi _{|V _k|}$ (indeed, $\theta_k^{\ast}\in\text{Aut}(\mathbb{P}^N)$) induced by $\theta_k$. Hence $\kappa^{(2)}(X,K_X +L, h_L)=\max_{k\in\mathbb{N}} \dim \text{Im} \Phi _{|V _k|}$.
     Now applying Lemma \ref{lema41kk12} and Lemma \ref{resp} respectively to the graded algebra $S^{(a a_0)}\simeq\bigoplus_{k\in\mathbb{N}} V_k$ of almost integral type, we then obtain that both $\kappa^{(2)}_{a a_0}$ and $\kappa^{(3)}_{a a_0}$ are equal to $\dim \Delta(S^{(a a_0)})$. Our claim follows in the end as $\kappa^{(2)}_{a a_0}=\kappa^{(2)}$ and $\kappa^{(3)}_{a a_0}=\kappa^{(3)}$, thanks to (\ref{harv}).
      \end{proof}

In the special case when $S$ is finitely generated, $$\kappa^{(2)}(X,K_X +L,h_L)=\kappa^{(3)}(X,K_X +L,h_L)$$ can be obtained immediately as mentioned in the previous subsection. However, it is still in question that whether $S$ defined as in (\ref{notj}) is finitely
generated or not. In {\cite{bojiezhou}}, the authors has reduced this problem to the study of the singularities of admissible Bergman metrics.

A dominating meromorphic map $f : X \dashrightarrow Y$ between a compact
complex manifold $X$ and a projective algebraic variety $Y$ is
determined by the subfield $K = Q (Y) \subset Q (X)$ up to bimeromorphic
equivalence. For such $f$, there exist two proper modifications $g' : X'
\rightarrow X$ and $h' : Y' \rightarrow Y$ such that $Y'$ is a projective
manifold and the induced map $f' : X' \dashrightarrow Y'$ is actually a
morphism. This is obtained by choosing $X'$ as a non-singular model (whose
existence is due to Hironaka, cf. Theorem 2.12 in {\cite{ueno00}}) of the
graph of the meromorphic mapping $X \dashrightarrow Y'$ and $g'$ as the
projection onto the factor $X$. In this way, $f'$ is called a representative
of $f$ or the function field $K$. Recall that $f'$ is a fiber space if and
only if $K$ is algebraically closed in $Q (X)$ (cf. Corollary 1.10 in
{\cite{ueno00}}).

We can now state the following main result of this subsection, whose main idea
comes from (the proof of) Theorem 1.11 in {\cite{mori1}}.

\begin{theorem}
  \label{below}Let $X$ be a compact complex manifold and $L \rightarrow X$ be
  a pseudoeffective line bundle equipped with a singular metric $h_L$ whose
  curvature current is semi-positive. Assume that $h_L$ has analytic
  singularities (see its definition in section 1) and $\kappa (X, K_X + L, h_L) \geq 0$. Let $f^{\flat} :
  X^{\flat} \rightarrow Y^{\flat}$ be a representative of
  \begin{equation}
    f (= \Phi_{|S_k |}) : X \dashrightarrow \tmop{Im} \Phi_{|S_k |}
    \backassign Y
  \end{equation}
  for any $k \in N (X, K_X + L, h_L)$ sufficiently large so that Proposition
  {\tmem{\ref{pdju}}} holds. Then $f^{\flat} : X^{\flat} \rightarrow
  Y^{\flat}$ is an analytic fiber space, $\kappa(X, K_X + L, h_L) =
  \dim Y^{\flat}$ and $\kappa (F, K_F + L_F, h_L |_F) = 0$, where
  $F $ is the very general (which means, outside a countable union of analytic
  subsets) fiber of $f^{\flat}$.
\end{theorem}

\begin{proof}
  According to Proposition \ref{alcl}, one may find $f^{\flat} :
  X^{\flat} \rightarrow Y^{\flat}$ as an algebraic fiber space and
  $\kappa^{(2)} (X, K_X + L, h_L) = \dim Y^{\flat}$. By the analytic
  singularities assumption on $h_L$, we see that
  by taking a resolution of singularities of $f^{\flat \ast } \varphi_{L}$ ($h_L = e^{- \varphi_{L}}$ locally) one may assume
  \begin{equation}
    f^{\flat \ast } \varphi_{L} \sim \alpha \sum \lambda_j \log |g_j |,
  \end{equation}
  where $\alpha \in \mathbb{R}_{> 0}$, $g_j$ are local generators of the
  invertible sheaves $\mathcal{O}_X (- D_j)$ and $D = \sum \lambda_j D_j$ is a
  normal crossing divisor on $X^{\flat}$  which also contains all the components of the exceptional divisors. Then
  $\mathcal{I}_k (f^{\flat \ast } \varphi_{L})$ is computed by
  \begin{equation}
    \mathcal{I}_k (f^{\flat \ast } \varphi_{L}) = \mathcal{O}_{X'} (-
    \sum ([k \alpha \lambda_j] - k + 1) D_j) \label{uski}
  \end{equation}
  and therefore locally free.

  Let $E_{l}$ denote the $\mathbb{}$Cartier divisor on $X^{\flat}$
  associated to the invertible sheaf $\mathcal{O}_{X^b} (l  (K_{X^{\flat}} +
  L^{\flat}) \otimes \mathcal{I}_l (f^{\flat \ast } h_{L}))$ and $l \in
  \mathbb{N}$. Let $H$ be a very ample divisor on $Y^{\flat}$, we claim that
  \begin{equation}
    f^{\flat \ast } \mathcal{O}_{Y^{\flat}} (H) \subset \mathcal{O}_{X^b}
    (E_{k}) \label{kdol}
  \end{equation}
  holds with our choice of $k$.

  Indeed, let $1, \sigma_1, \ldots, \sigma_N$ be a basis of $H^0 (Y^{\flat},
  \mathcal{O}_{Y^{\flat}} (H) \mathcal{}) = \{\sigma \in Q (Y^{\flat}) ;
  (\sigma) + H \geq 0\}$. Here, for a Cartier divisor $D$, $'' D \geq 0''$
  means all the coefficients of $D$ is non-negative, i.e. $D$ is effective. By
  our choice of $k$ so that $Q ((S_0 [S_k])) = Q ((X , K_{X } + L , h_L))
  \supset Q (Y^{\flat})$ (cf. Proposition \ref{pdju}), there exist $N > 0$ and
  $\xi_0, \ldots, \xi_N \in H^0 (X^{\flat}, \mathcal{O}_{X^{\flat}} (E_{k}))$ such
  that $\xi_0 \neq 0$ and $\sigma_j = \xi_j / \xi_0$, $j = 0, \ldots N$. Hence
  we get
  \begin{equation}
    0 \leq (\xi_j) + E_{k} = f^{\flat \ast } ((\sigma_j) + H) + ((\xi_0)
    + E_{k} - f^{\flat \ast } H) .
  \end{equation}
  Since $H$ is very ample, we have \ $\bigcap_{j = 1}^N f^{\flat \ast } ((\sigma_j) +
  H) = \varnothing$; thus $(\xi_0) + E_{k} - f^{\flat \ast } H \geq 0$.
  Hence multiplying by $\xi_0$ gives $f^{\flat \ast } \mathcal{O}_{Y^{\flat}}
  (H) \subset \mathcal{O}_{X^b} (E_{k})$, which confirms our claim.

  Since $\kappa^{(2)} (X, K_X + L, h_L) \geq 0$, we have $H^0 (X^{\flat}, k_1
  (K_{X^{\flat}} + L^{\flat}) \otimes \mathcal{I}_{k_1} (f^{\flat \ast } h_L))
  \neq 0$ for some $k_1 \in N (X, K_X + L, h_L)$.
It follows from the restriction map $$H^0
  (X^{\flat}, k_1 (K_{X^{\flat}} + L^{\flat}) \otimes \mathcal{I}_{k_1}
  (f^{\flat \ast } h_L)) \rightarrow H^0 (X_y^{\flat}, k_1 (K_{X_y^{\flat}} +
  L_y^{\flat}) \otimes \mathcal{I}_{k_1} (f^{\flat \ast } h_L) |_{X^{\flat}_y})$$ that $\kappa^{(2)} (F,
  K_F + L_F, h_L |_F) \geq 0$ for general fibers $F=X^{\flat}_{y}$, thanks to the fact that
  $\mathcal{I}_{k_1} (f^{\flat \ast } h_L |_{X_y}) = \mathcal{I}_{k_1}
  (f^{\flat \ast } h_L) |_{X_y}$ holds generally. Therefore, to prove
  $\kappa^{(2)} (F, K_F + L_F, h_L |_F) = 0$ holds for very
  general $F$, it only suffices to show the coherent sheaf
  \begin{equation}
    \mathcal{F}_a \assign f^{\flat}_{\ast} \mathcal{O} (a (K_{X^{\flat}} +
    L^{\flat}) \otimes \mathcal{I}_{a } (f^{\flat \ast } h_L)) = :
    f^{\flat}_{\ast} \mathcal{O} (E_a)
  \end{equation}
  has rank at most $1$ for every $a \in N (X, K_X + L, h_L) .$

  Thanks to (\ref{kdol}), there exist injections
  \begin{equation}
    \mathcal{O} (E_a) \otimes f^{\flat \ast } \mathcal{O}_{Y^{\flat}} (b H)
    \subset \mathcal{O}  (E_a) \otimes \mathcal{O} (b E_{k}) \label{8ik}
  \end{equation}
  for any $b \in \mathbb{N}$ (since $\mathcal{O} (E_a)$ is locally free). By
  the basic relation $ \mathcal{I}_k (f^{\flat \ast
  } h_L)^{\otimes b} \cdot \mathcal{I}_a (f^{\flat \ast } h_L) \subset
  \mathcal{I}_{a + k b} (f^{\flat \ast } h_L)$ and the left exactness of the
  functor $f^{\flat}_{\ast}$ , (\ref{8ik}) also induces injections
  \begin{equation}
    \mathcal{F }_a \otimes \mathcal{O} (b H) \hookrightarrow f^{\flat}_{\ast}
    \mathcal{O} ((a + k b) (K_X + L) \otimes \mathcal{I}_k (f^{\flat \ast }
    h_{L})^{\otimes b} \otimes \mathcal{I} _a (f^{\flat \ast } h_L))
    \hookrightarrow \mathcal{F}_{a + k b} \label{leftex}
  \end{equation}
  for all $b$. One may obtain a commutative diagram
  \begin{equation}
    \text{} \begin{array}{ccc}
      H^0 (Y^{\flat}, \mathcal{F }_a \otimes \mathcal{O} (b H)) &
      \hookrightarrow & H^0 (Y^{\flat}, \mathcal{F}_{a + k b}) = H^0
      (X^{\flat}, \mathcal{O} (E_{a + k b}))\\
      \downarrow \beta &  & \downarrow \gamma\\
      \mathcal{F}_a \otimes \mathcal{O} (b H) \otimes_{\mathcal{O}_{Y^b}}
      \mathcal{O}_{Y^{\flat}, y} / \mathfrak{m }_y & \hookrightarrow &
      \mathcal{F}_{a + k b} \otimes_{\mathcal{O}_{Y^b}}
      \mathcal{O}_{Y^{\flat}, y} / \mathfrak{m }_y
    \end{array}
  \end{equation}
  from (\ref{leftex}). Since $Q ((X, K_X + L, h_L)) = Q (Y)$ (cf. Proposition \ref{pdju}),
  $\tmop{Im} \gamma$ is a vector space of dimension $\leq 1$. Since $\beta$ is
  surjective for $b \gg 1$, thus $\mathcal{F}_a$ has rank at most $1$.
\end{proof}

\section{ Proof of Corollary \ref{dio}}\label{sect4}

In this section, for a holomorphic vector bundle $E$ (resp. coherent sheaf
$\mathcal{E}$), the notation $\mathbb{P} (E)$ (resp. $\mathbb{P}
(\mathcal{E})$) stands for the projective fiber space containing all $1$-dim
quotients of fibers of $E$ (resp. the set of all quotient invertible sheaves
of $\mathcal{E}$, cf. (2.8) and (2.9) in {\cite{ueno00}}).

To give the proof of Corollary \ref{dio}, we first consider some results on
generalized Kodaira dimensions from the viewpoint of birational geometry.
The following proposition is a slight modification of Theorem 5.11 in
{\cite{ueno00}}. For completeness, we give details of its proof here.

\begin{proposition}
  \label{lemmakey}\label{prop4.3}Let $f : X \rightarrow Y$ be a surjective and
  proper holomorphic map between two compact complex manifolds $X$ and $Y$ with connected fibers. Let $L \rightarrow X$ be a  $\mathbb{Q}$-line bundle equipped with a singular metric $h_L$ with semi-positive
  curvature current. Then there exists a set $\Sigma \subset Y$ whose complement has measure zero, such that
  \begin{equation}
    {\kappa} (X, K_{X} + L, h_L) \leq {\kappa} (F, K_F + L
    |_F, h_L |_F) + \dim Y, \nobracket \nobracket
  \end{equation}
  where $F\subset f^{-1}(\Sigma)$ denotes the sufficiently general fiber of $f$.
\end{proposition}

\begin{proof}
  Let us assume $\kappa (X, K_{X } + L, h_L) \neq - \infty$ without loss of generality. Then for
  sufficiently large and divisible $k \in \mathbb{N}$, we set
  \begin{equation}
    \mathcal{F}_k \assign f_{\ast} (k (K_{X } + L) \otimes \mathcal{I}_k
    (h_L)) \neq 0,
  \end{equation}
  which is coherent over $Y$ by Grauert's direct image
  theorem. Moreover, by Grauert's theorem (cf. Theorem 10.6 and Theorem 10.7 in \cite{bertin}), there exists a Zariski open dense subset $U_k \subset
  Y_0$ such that $\mathcal{F}_k |_{U_k}$ is locally free and the following
  base change property
  \begin{equation}
    \mathcal{F }_{k, y}\otimes \mathcal{O}_{Y,y} /  \mathfrak{m}_y = H^0 (X_y, (k K_{X_y} {+ k L|_{X_y}} ) \otimes
    \mathcal{I}_k (h_L) |_{X_y}) \label{26}
  \end{equation}
  holds for $y \in U_k$. As mentioned in section \ref{sect2}, we already
  obtain that
  \begin{equation}
    \mathcal{I}_k (h_L) |_{X_y} = \mathcal{I}_k (h_L |_{X_y}) \label{iud}
  \end{equation}
  for all $y \in \Sigma_k$ where $\Sigma_k \subset Y$ is a subset with full
  measure. Let $\mathbb{P} (\mathcal{F}_k)$ be the projective fiber space
  associated to $\mathcal{F }_k$, then $f|_{f^{- 1} (U_k)}$ factors through a meromorphic map
  $h_k : X \dashrightarrow \mathbb{P} (\mathcal{F}_k)$ and a morphism $g_k : \mathbb{P}
  (\mathcal{F}_k) \rightarrow Y$, i.e. $f|_{f^{- 1} (U_k)} =
  g_k \circ h_k$ (cf. (2.10) in {\cite{ueno00}}).

  On one side, by (\ref{26}) and (\ref{27}), $h_k |_{X_y}$ for $y
  \in \Sigma_k  \bigcap U_k$ is given by the following meromorphic map
  \begin{equation}
    \Phi_{k, y} : X_y \dashrightarrow \mathbb{P} H^0 (X_y, (k K_{X_y} + k L
    |_{X_y}) \otimes \mathcal{I}_k (h_L |_{X_y})) \label{pijn}
  \end{equation}
  with respect to a \emph{sublinear} series of the divisor $k K_{X_y} + L|_{X_y}$
  (cf. (2.10) in {\cite{ueno00}}). It follows that
  \begin{equation}
    \dim h_k (X) = \dim Y + \dim \tmop{Im} \Phi_{k, y}
    \label{kol;}
  \end{equation}
  holds for $y \in \Sigma_k \bigcap U_k'$, where $U_k' \subset U_k$ is another
  Zariski open dense subset (cf. Remark \ref{atten}).

  On the other side, identifying
  \begin{equation}
    H^0 (Y, \mathcal{F}_k) = H^0 (X, k (K_{X } + L) \otimes \mathcal{I}_k
    (h_L)) \label{210}
  \end{equation}
  we obtain a meromorphic map
  \begin{equation}
    h : \mathbb{P} (\mathcal{F }_k) \dashrightarrow \mathbb{P} H^0
    (\mathbb{P} (\mathcal{F }_k), \mathcal{O }_{\mathbb{P} (\mathcal{F }_k)}
    (1)) = \mathbb{} \mathbb{P} H^0 (Y, \mathcal{F}_k) = \mathbb{P }^{N_k}
  \end{equation}
  such that $h \circ h_k$ is the Kodaira meromorphic mapping $\Phi_k$ on $X$ with respect to
  \textbar$H^0 (X, k (K_{X } + L) \otimes \mathcal{I}_k (h_L)) |$, here
  $N_k = \dim H^0 (X, k (K_{X } + L) \otimes \mathcal{I}_k (h_L)) - 1$.

  Therefore, there is a generically surjective (cf. Definition 2.6 in
  {\cite{ueno00}}) meromorphic map $h : h_k (X) \dashrightarrow \tmop{Im}
  \Phi_k .$ Taking $k$ sufficiently large and divisible, we conclude from
  (\ref{kol;}) and Proposition \ref{eqji} that
  \[ {\kappa} (X, K_{X} + L, h_L) = \dim \tmop{Im} \Phi_k \leq \dim
     h_k (X) = \dim Y + \dim \tmop{Im} \Phi_{k, y} \leq \]
  \begin{equation}
    \dim Y + {\kappa} (X_y, K_{X_y} + L|_{X_y}, h_L |_{X_y}) \nobracket
  \end{equation}
  for $y \in \Sigma$.
\end{proof}

\begin{remark}
  \label{atten}We might give some details of (\ref{kol;}) as follows.

Let $\Gamma_k \assign \overline{h_k (f^{- 1} (U_k))}$ be the closure of the
  image of $h_k$ in $\mathbb{P} (\mathcal{F}_k)$, which is a complex analytic
  subspace of $\mathbb{P} (\mathcal{F}_k)$ (or the projection of the graph of
  $h_k$ onto the second factor $\mathbb{P} (\mathcal{F}_k)$). When considering the generically surjective
  morphism $\tilde{g}_k : \Gamma_k \rightarrow Y$ induced by the restriction
  of $g_k$ on $\Gamma_k$, there is a Zariski open dense subset ${U'_k}
  \subset U_k$ such that $\dim \tilde{g}_k^{- 1} (y) = \dim \Gamma_k - \dim Y$
  is constant for any $y \in U_k'$. Since $h_k |_{X_y}$ is given by
  (\ref{pijn}) for any $y \in \Sigma_k  \bigcap U'_k$, we conclude that
  \[ \dim \tmop{Im} \Phi_{k, y} = \dim h_k (X_y) = \dim \tilde{g}_k^{- 1} (y)
     = \dim \Gamma_k - \dim Y \]
 holds for sufficiently general $y$.
\end{remark}

The following lemma can be seen as a quantitative version of Lemma 3.1 in
{\cite{juanyong-wang}}.

\begin{lemma}
  \label{lem4.4}Let $f : X \rightarrow Y$ be a surjective holomorphic map
  between a compact complex manifold $X$ and a projective manifold
  $Y$ with connected fibers. Let $L \rightarrow X$ be a  $\mathbb{Q}$-line bundle equipped with a singular metric $h_L$ with semi-positive
  curvature current and $A  \rightarrow Y$ be an ample  $\mathbb{Q}$-line bundle. If
  $\kappa (X, K_{X / Y} + L, h_L) \geq 0$ and $h_L$ has analytic
  singularities, then
  \begin{equation}
    {\kappa} (X, K_{X / Y} + L + f^{\ast} A , h_L) = {\kappa}
    \left( F, K_F + L \left|_F, h_L |_F) + \dim Y, \label{clai} \right.
    \right.
  \end{equation}
  where $F = X_y$ for sufficiently general $y \in Y$.
\end{lemma}

\begin{proof}
  By multiplying a sufficiently large and divisible integer $k_0$, we may
  assume that $H^0 (X, (k_0 K_{X / Y} + k_0 L) \otimes \mathcal{I}_{k_0}
  (h_L)) \neq 0$ and $k_0 A $ is a very ample line bundle over $Y $.

  Now we consider the meromorphic map
  \begin{equation}
    \Phi \assign \Phi_{| V_{k_0} | \nobracket \nobracket} : X \dashrightarrow
    \mathbb{P} V_{k_0}  \label{214}
  \end{equation}
  by enlarging $k_0$ so that Proposition \ref{pdju} holds, where $\mathbb{}
  V_{k_0} = H^0 (X, (k_0 K_{X / Y} + k_0 L + k_0 f^{\ast} A ) \otimes
  \mathcal{I}_{k_0} (h_L))$. Upon a modification $\mu : X' \rightarrow X$,
  where $X'$ is obtained as a smooth model of the graph of the meromorphic map $\Phi_{| V_{k_0} | \nobracket
  \nobracket}$, we obtain that $$\Phi' \assign \Phi'_{| V'_{k_0} |
  \nobracket \nobracket} : X' \rightarrow \tmop{Im} \Phi'_{| V'_{k_0} |
  \nobracket \nobracket} =\tmop{Im} \Phi_{| V_{k_0} |
  \nobracket \nobracket} \subset \mathbb{P} V'_{k_0}=\mathbb{P} V_{k_0}$$ is an analytic fiber
  space (i.e. $\Phi'$ is a representative of $\Phi$, cf. Theorem \ref{below}), where
  $$
  V'_{k_0}:=H^0(X',k_0(K_{X' / Y} + L' + \mu^{\ast} f^{\ast} A)\otimes \mathcal{I}_{k_0}(h'_L))=V_{k_0}, (L',h'_L)=(\mu^{*}L,\mu^{*}h_L).
  $$
  We also obtain
  $$ \tilde{\kappa} (F, K_F + L |_F, h_L |_F)=\tilde{\kappa} (F', K_F' + L' |_{F'}, h'_L |_{F'}) $$
  \begin{equation}
    \label{jj} \tilde{\kappa} (X, K_{X / Y} + L + f^{\ast} A, h_L) =
    \tilde{\kappa} (X', K_{X' / Y} + L' + (f')^{\ast} A, h_{L'})
  \end{equation}
  by Proposition \ref{oo}, where $F$ (resp. $F'$) is the general fiber of $\Phi$ (resp. $\Phi'$). Additionally we have $\dim \tmop{Im} \Phi' =\tilde{ \kappa }(X', K_{X' / Y} + L + (f')^{\ast}A, h_L)$. Viewing $H^0 (Y, k_0 A) = H^0 (X', k_0 (f')^{\ast} A)
  \hookrightarrow V'_{k_0}$, the following commutative diagram
  \begin{equation}
    \text{} \begin{array}{ccc}
      \mathbb{P} V'_{k_0}& \dashrightarrow & \mathbb{P} H^0 (Y, k_0 A) \\
      \Phi' \uparrow &  & \uparrow \\
      X'& \overset{f'= f \circ \mu}{\longrightarrow}  & Y
    \end{array} \label{216}
  \end{equation}
  holds. As a consequence of (\ref{216}), the general fiber $G'$ of $\Phi'$ is
  contracted by $f'$, hence we obtain a fiber space
   $\Phi' |_{F'} : F' \rightarrow \tmop{Im} \Phi' |_{F'}$ whose general
  fiber is $G'$ (note $F'$ is connected and $G'$ is also connected thanks to Proposition \ref{alcl}). It is already known that $\mathcal{I}_{k_0}
  (h_{L'}) |_{F'} = \mathcal{I}_{k_0} (h_{L'} |_{F'})$ holds for sufficiently general
  fibers $F'$. Therefore, $\Phi' |_{F'}$ is determined by a  \emph{sublinear} series of $| \mathbb{} H^0 (F', (k_0 K_{F'} + k_0{ L'}|_{F'})
  \otimes \mathcal{I}_{k_0} (h_{L'} |_{F'})) |$, which contains elements that can be extended to $V'_{k_0}$.

  On one hand, we obtain
  \[ \tilde{\kappa} (F', K_{F'} + L |_{F'}, h_{L'} |_{F'}) \geq \dim \tmop{Im} \Phi' |_{F'} =
     \dim F' - \dim G' = \nobracket \nobracket \]
  \begin{equation}
    \dim \tmop{Im} \Phi' - \dim Y = \tilde{\kappa} (X', K_{X' / Y} +
    L' + (f')^{\ast} A,  h_{L'}) - \dim Y. \label{27}
  \end{equation}
  On the other hand, applying Proposition \ref{lemmakey} to $\Phi' |_{F'}$ we get
$$
    \tilde{\kappa}( F', K_{F'} + L' |_{F'}, h_{L'} |_{F'}) = \tilde{\kappa}
    ( F', K_{F'} + (L' + (f')^{\ast} A) |_{F'}, h_{L'} |_{F'}) \leq
$$
\[
    \dim\tmop{Im} \Phi' |_{F'} + {\tilde{\kappa} (G', K_{G'}
    + (L' + (f')^{\ast} A) |_{G'}, h_{L'} |_{G'}) \nobracket \nobracket}= \]
\begin{equation}\label{pn}
     \dim
    \tmop{Im} \Phi' |_{F'} = \tilde{\kappa} (X', K_{X' / Y} +  L' +
    (f')^{\ast} A , h_{L'}) - \dim Y.
\end{equation}
  To see why ${\tilde{\kappa} (G', K_{G'}
    + (L' + (f')^{\ast} A) |_{G'}, h_{L'} |_{G'}) \nobracket \nobracket}=0$ in (\ref{pn}) (where we have used the
  assumption on $h_L$), one can just apply Theorem \ref{below} to the generalized Iitaka fibration $\Phi'$. In the end, our claim of (\ref{clai}) follows easily by (\ref{27}),
  (\ref{pn}) and (\ref{jj}) and Proposition \ref{eqji}.
\end{proof}

We are now in position to accomplish the proof of Corollary \ref{dio}.

\begin{proof}
  (Proof of Corollary \ref{dio}) Let us first fix an ample bundle $A$ over $Y$
  such that $A - K_Y$ is sufficiently ample and satisfies the assumption
  (\ref{dddd}) in Lemma \ref{lemma11}. Since $K_Y$ is a big line bundle, we
  may write
  \begin{equation}
    k K_Y = A + E
  \end{equation}
  for some large $k \in \mathbb{N}$ with an effective line bundle $E$. We may assume that $f_{\ast} ((2 k K_{X /
  Y} + 2 k L) \otimes \mathcal{I}_{2 k} (h_L) \mathcal{}) \neq 0$, otherwise
  there is nothing to prove. Then
  Lemma \ref{lemma11} already implies that ${\kappa} \left( X, K_{X/Y} + L +
  \frac{1}{2 k} f^{\ast} A, h_L \right) \geq 0$, hence one can apply Lemma
  \ref{lem4.4} to obtain that
  \[ {\kappa} (X, K_X + L, h_L) = {\kappa} (X, K_{X / Y} + L +
     f^{\ast} K_Y, h_L) \geq \]
  \[ {\kappa} (X, (K_{X / Y} + L + \frac{1}{2 k} f^{\ast} A) +
     \frac{1}{2 k} f^{\ast} A, h_L) = {\kappa} (F, K_F + L |_F, h_L |_F)
     + \dim Y \nobracket \]
  holds for general $F$. For the other direction, it will just be an immediate consequence
  of Proposition \ref{lemmakey}.
\end{proof}

\

\noindent\textbf{Acknowledgements:} The authors would sincerely thank the anonymous referee(s) for his/her suggestions to the original version of this paper, especially for Remark \ref{refer}. The first author would like to thank Prof. Z. W. Wang for some fruitful discussion about this paper and Dr. J. Y. Wang for introducing some background about the Iitaka conjecture and some of his work on this topic.

\end{document}